\documentclass[12pt]{article}
\usepackage{amsmath}
\usepackage{fullpage}
\usepackage{amsfonts}
\usepackage{amssymb}
\usepackage{graphicx}
\usepackage{verbatim}
\immediate\write18{texcount -tex -sum  \jobname.tex > \jobname.wordcount.tex}
\usepackage[ruled,vlined]{algorithm2e}
\usepackage[titles]{tocloft}
\usepackage{mathrsfs}
\usepackage{bm}
\usepackage{enumitem}
\usepackage{amsthm}
\usepackage{chngcntr}
\usepackage{titlesec}
\usepackage{stfloats}
\usepackage{hyperref}
\hypersetup{linktocpage}
\hypersetup{
    colorlinks,
    citecolor=red,
    filecolor=black,
    linkcolor=blue ,
    urlcolor=black
}

\def\X{S}
\def\sO{\mathscr{O}}

\def\sA{\mathscr{A}}
\def\scrQ{\ensuremath{\mathscr{Q}}}

\def\frkJ{\mathfrak{J}}
\def\frkL{\mathfrak{L}}

\def\fB{{\mathfrak B}}
\def\fN{{\mathfrak N}}

\def\fp{\mathfrak{P}}

\def\A{\mathfrak{A}}

\def\bz{{{\bm 0}}}
\def\Lb{\textit{\textbf{L}}}
\def\Xb{\textit{\textbf{X}}}

\def\Thetab{\bm{\Theta}}
\def\thetab{\bm{\vartheta}}
\def\mA{{\bm A}}
\def\ub{{\bm u}}

\def\xb{{\bm x}}
\def\yb{{\bm y}}

\def\ab{\textit{\textbf{a}}}

\def\pjb{\overline{P_{j}}}

\def\vt{\vartheta}

\def\D{\Delta}
\DeclareMathOperator{\GL}{GL}

\DeclareMathOperator{\htt}{ht}
\DeclareMathOperator{\sing}{sing}

\DeclareMathOperator{\jac}{jac}
\DeclareMathOperator{\grad}{grad}

\DeclareMathOperator{\reg}{reg}
\DeclareMathOperator{\rk}{rank}

\def\pa{\partial}

\def\dt{s}
\newcommand{\ZZ}{{\mathbb{Z}}}
\def\C{\mathbb{C}}
\def\Q{\mathbb{Q}}
\def\R{\mathbb{R}}

\def\pr{\mathbb{P}}

\def\bbm{\begin{bmatrix}}
\def\ebm{\end{bmatrix}}
\def\bmat{\begin{matrix}}
\def\emat{\end{matrix}}

\def\gi{\Gamma_i}
\newtheorem{theorem}{Theorem}[section]
\newtheorem{corollary}[theorem]{Corollary}
\newtheorem{lemma}[theorem]{Lemma}
\newtheorem{ex}[theorem]{Example}

\newtheorem{prop}[theorem]{Proposition}

\newtheorem{remark}[theorem]{Remark}

\providecommand{\keywords}[1]
{
  \smallskip\noindent\small	
  \textbf{\textbf{Keywords---}} #1
}
\title{Bit complexity for computing one point in each connected component of a smooth real algebraic set}
\author{Jesse Elliott, Mark Giesbrecht, \'Eric Schost  \\
  \small David R. Cheriton School of Computer Science, University of Waterloo, On, Canada \\
}
\date{} 
\begin{document}
\maketitle
%
%
%
\begin{abstract}
We analyze the bit complexity of an algorithm for the computation of
at least one point in each connected component of a smooth real
algebraic set. This work is a continuation of our analysis of the
hypersurface case ({\em On the bit complexity of finding points in
  connected components of a smooth real hypersurface}, ISSAC'20). In
this paper, we extend the analysis to more general cases.

Let $F=(f_1,\hdots, f_p)$ in $\mathbb{Z}[X_1, \hdots , X_n]^p$ be a
sequence of polynomials with $V = V(F) \subset \C^n$ a smooth and
equidimensional variety and $\langle F \rangle \subset \C[X_1, \hdots
  , X_n]$ a radical ideal. To compute at least one point in each
connected component of $V \cap \mathbb{R}^n$, our starting point is an
algorithm by Safey El Din and Schost ({\em Polar varieties and
  computation of one point in each connected component of a smooth
  real algebraic set}, ISSAC'03). This algorithm uses random changes
of variables that are proven to generically ensure certain desirable
geometric properties. The cost of the algorithm was given in an
algebraic complexity model; here, we analyze the bit complexity and
the error probability, and we provide a quantitative analysis of the
genericity statements. In particular, we are led to use Lagrange
systems to describe polar varieties, as they make it simpler to
rely on techniques such as weak transversality and an effective 
Nullstellensatz.

\end{abstract}
%


\section{Introduction}

\paragraph*{Background and problem statement.}
Computing one point in each connected component of a real algebraic
set $S$ is a basic subroutine in real algebraic and semi-algebraic
geometry; it is also useful in its own right, since it allows one to
decide if $S$ is empty or not. 

We consider the case where $S$ is given as $S=V \cap \R^n$, where
$V=V(F) \subset \C^n$ is a complex algebraic set defined by a sequence
of polynomials $F = (f_1,\hdots,f_p)$ in
$\ZZ[X_1,\dots,X_n]^p$. Algorithms for this task have been known for
decades, and their complexity is to some extent well
understood. Suppose that all $f_i$'s have degree at most $d$, and
coefficients of bit-size at most $b$. Without making any assumption on
these polynomials, the algorithm given
in~\cite[Section~13.1]{BaPoRo03} solves our problem using $d^{O(n)}$
operations in $\Q$; in addition, the output of the algorithm is
represented by polynomials of degree $d^{O(n)}$, with coefficients of
bit-size $hd^{O(n)}$. The key idea behind this algorithm goes back
to~\cite{GrVo88}: sample points are found through the computation of
critical points of well-chosen functions on $V$.

The number of connected components of $V$ admits the lower bound
$d^{\Omega(n)}$, so up to polynomial factors this result is
optimal. However, due to the generality of the algorithm, the constant
hidden in the exponent $O(n)$ in its runtime turns out to be rather
large: the algorithm relies on infinitesimal deformations, that affect
runtime non-trivially.

In this paper, we will work under the additional assumption that
$V=V(f_1,\dots,f_p)$ is a {\em smooth} complex algebraic set,
equidimensional of dimension $\delta=n-p$, and that $f_1,\dots,f_p$
generate a radical ideal (we explain these terms in the next
section). We place ourselves in the continuation of the line of work
initiated by~\cite{BaGiHeMb97}: that reference deals with cases where
$V$ is a smooth hypersurface and $V \cap \R^n$ is compact, pointing
out how {\em polar varieties} (that were introduced in the 1930's in
order to define characteristic classes~\cite{Piene78,Teissier88}) can
play a role in effective real geometry. This paper was extended in
several directions: to $V$ being a smooth complete intersection, still
with $V\cap \R^n$ compact~\cite{BaGiHeMb01}, then without the
compactness assumption~\cite{EMP,BaGiHePa05}; the smoothness
assumption was then partly dropped in~\cite{BaGiHe14,BaGiHeLePa12}.

Our starting point is the algorithm in~\cite{EMP}, whose assumptions
are slightly more general than ours ($V$ is not required to have
dimension $\delta=n-p$). In the cases we consider in this paper, its
runtime is ${ n \choose p}{}^{2+o(1)}d^{(4+o(1))n}$ operations in
$\Q$.  As with many results in this vein, the algorithm is randomized,
as we need to assume that we are in generic coordinates; this is done
by applying a random change of coordinates prior to all
computations. In addition, the algorithm relies on procedures for
solving systems of polynomial equations that are themselves
randomized.  Altogether, we choose $n^{O(1)}$ random vectors, each of
them in an affine space of dimension $n^{O(1)}$; every time a choice
is made, there exists a hypersurface of the parameter space that one
has to avoid in order to guarantee success. In this paper, we revisit
this algorithm, modify it in part, and give a complete analysis of its
probability of success and its bit complexity.

This work is a continuation of the analysis of the hypersurface case
that we gave in~\cite{ElGiSc20} (that is, the case $p=1$). A very
useful property in the hypersurface case is that polar varieties can
be described by straightforward equations (the partial derivatives of
the input polynomial) that form a regular sequence, at least in
generic coordinates. In higher codimension, this is not the case
anymore: the natural description of polar varieties now involves
minors of the Jacobian matrix of the input equations (this is the
approach used in~\cite{EMP}). The resulting equations are in general
not a complete intersection anymore, which makes it impossible to
extend directly several arguments we used in~\cite{ElGiSc20}.

Our solution is to use a description of polar varieties by means of
so-called Lagrange equations. These equations are complete
intersections (in generic coordinates), but they involve more
variables. As such, they describe algebraic sets that cover polar
varieties; we will discuss in detail the relationship between these
two presentations, using in particular several results
from~\cite{BaGiHeSaSh10,TWT}.

\paragraph*{Data structures.} 
The output of the algorithm is a finite set in $\overline{\Q}{}^n$. To
represent it, we rely on a widely used data structure based on
univariate
polynomials~\cite{Kronecker82,Macaulay16,GiMo89,GiHeMoPa95,ABRW,GiHaHeMoMoPa97,GiHeMoMoPa98,Rouillier99}.
For a zero-dimensional algebraic set $S \subset \C^n$ defined over
$\Q$, a {\em zero-dimensional parameterization}
$\scrQ=((q,v_1,\dots,v_n),\lambda)$ of $S$ consists in polynomials
$(q,v_1,\dots,v_n)$, such that $q\in \Q[T]$ is monic and squarefree,
all $v_i$'s are in $\Q[T]$ and satisfy $\deg(v_i) < \deg(q)$, and in a
$\Q$-linear form $\lambda$ in variables $X_1,\dots,X_n$, such that
\begin{itemize}
\item $\lambda(v_1,\dots,v_n)=T q' \bmod q$;
\item we have the equality
  $S=\left \{\left(
      \frac{v_1(\tau)}{q'(\tau)},\dots,\frac{v_n(\tau)}{q'(\tau)}\right
    ) \ \mid \ q(\tau)=0 \right \}.$
\end{itemize}
The constraint on $\lambda$ says that the roots of $q$ are the values
taken by $\lambda$ on $S$. The parameterization of the coordinates by
rational functions having $q'$ as a denominator goes back
to~\cite{Kronecker82,Macaulay16}: as pointed out in~\cite{ABRW}, it
allows one to control precisely the size of the coefficients of
$v_1,\dots,v_n$.

\paragraph*{Main result.} To state our main result, we need to define 
the {\em height} of a rational number, and of a polynomial with
rational coefficients.

The {\em height} of a non-zero $a=u/v \in \Q$ is the maximum of
$\ln(|u|)$ and $\ln(v),$ where $u \in \mathbb{Z}$ and $v \in
\mathbb{N}$ are coprime. For a polynomial $f$ with rational
coefficients, if $v \in \mathbb N$ is the minimal common denominator
of all non-zero coefficients of $f$, then the \textit{height}
$\htt(f)$ of $f$ is defined as the maximum of the logarithms of $v$
and of the absolute values of the coefficients of $vf$.

\begin{theorem}\label{theo:main}
  Let $F= (f_1,\hdots,f_p)\in\ZZ[X_1,\hdots,X_n]^p$ be a sequence of
  polynomials with $\deg(f_i) \leq d$ and $\htt(f_i) \leq b$. Suppose
  that the ideal generated by $f_1,\dots,f_p$ is radical and that
  $V=V(F) \subset \C^n$ is smooth and equidimensional of dimension
  $n-p$. Also suppose that $0 < \epsilon < 1.$ 

  There exists a randomized algorithm that takes $F$ and $\epsilon$ as
  input and produces $n-p+1$ zero-dimensional parameterizations, the union
  of whose zeros includes at least one point in each connected
  component of $V(F) \cap \R^n$, with probability at least
  $1-\epsilon$. Otherwise, the algorithm either returns a proper
  subset of the points, or FAIL.  In any case, the algorithm uses
  \[
  O^{\sim}(d^{3n+2p+1}\log(1/\epsilon)(b + \log(1/\epsilon)))
  \]
  bit operations. The polynomials in the output have degree at most
  $d^{n+p},$ and height
  \[
  O^{\sim}(d^{n+p+1}(b + \log(1/\epsilon))).
  \]
\end{theorem}

Here we assume that $F$ is given as a sequence of polynomials in dense
representation.  Following references such
as~\cite{GiHeMoPa95,GiHaHeMoMoPa97,GiHeMoMoPa98,BaGiHeMb97,EMP}, it
would be possible to refine the runtime estimate by assuming that $F$
is given by a {\em straight-line program} (that is, a sequence of
operations $+,-,\times$ that takes as input $X_1,\dots,X_n$ and
evaluates $F$). Any polynomial of degree $d$ in $n$ variables can be
computed by a straight-line program that does $O(d^n)$ operations:
evaluate all monomials of degree up to $d$ in $n$ variables, multiply
them by their respective coefficients and sum the results. However,
some inputs may be given by a shorter straight-line program, and the
algorithm would actually benefit from this.

The algorithm itself is rather simple. To describe it, we need to
define {\em polar varieties}, which will play a crucial role in this
paper. Let $V=V(F)$, for $F=(f_1,\dots,f_p)$ as in the theorem. For $i
\in \{1,\hdots,n-1\},$ denote by $\pi_i:\C^n \rightarrow \C^i$ the
projection $(x_1,\hdots,x_n) \mapsto (x_1,\hdots,x_i)$.  The $i$-th
\textit{polar variety} \[W(i,F) := \{\xb \in V~|~\dim \pi_i(T_\xb
V) < i\}\] is the set of critical points of $\pi_i$ on $V$.  We will
recall below that it is defined by the vanishing of all $p$-minors
$M_{i,1},\dots,M_{i,S_i}$ of the last $n-i$ columns of the Jacobian
matrix of $F$, together with the equations $F$ themselves (here, $S_i$
is simply the binomial number $\binom{n-i}{p}$).

In general, we cannot say much about the geometry of $W(i,F)$, but if
we apply a generic change of coordinates $\mA$ to $F$, then $W(i,F)$
is known to be equidimensional of dimension $(i-1)$ or
empty~\cite{BaGiHeMb97,BaGiHeSaSh10,TWT}, and to be in so-called {\em
  Noether position}~\cite{EMP} (background notions in algebraic
geometry are in~\cite{Mumford76,Shafarevich77,ECA}; we will recall key
definitions). If this is the case, the algorithm in~\cite{EMP} chooses
arbitrary $\sigma_1,\dots,\sigma_{n-1}$ in $\Q$ and solves the systems
defined by
\begin{equation}\label{eq:syst1}
  X_1-\sigma_1 = \dots = X_{i-1}-\sigma_{i-1} = f_1 = \cdots = f_p = M_{i,1} = \cdots = M_{i,S_i} = 0
\end{equation}
for $i=1,\dots,n-p+1$.  They all admit finitely many solutions, and
Theorem~2 in~\cite{EMP} proves that the union of their solution sets
contains one point on each connected component of $V \cap \R^n$.

One of our contributions is to analyze precisely what conditions on
the change of coordinates $\mA$ guarantee success. This is done by
revisiting the key ingredients in the proofs given
in~\cite{BaGiHeSaSh10,EMP}, and giving quantitative versions of these
results, bounding the degrees of the hypersurfaces we have to avoid.

We actually do not solve the equations~\eqref{eq:syst1}, since the
(large) number of minors $S_i$ makes this analysis difficult.  Instead,
we replace~\eqref{eq:syst1} by equations involving Lagrange
multipliers. Proving correctness requires us to guarantee further
genericity properties, but once this is done, we can rely on the algorithm
in~\cite{SH} to solve these equations, for which a complete bit
complexity analysis is available.

\paragraph*{Further work.}
This paper is an extension of~\cite{ElGiSc20}, where the analysis was
done for the hypersurface case. In addition, this work should also be
seen as a step toward the analysis of further randomized algorithms in
real algebraic geometry.  In particular, randomized algorithms for
deciding {\em connectivity queries} on smooth, compact algebraic sets
have been developed in a series of papers
\cite{SchostMohabBabySteps2011,SchostMohabBabySteps2014}, and could be
revisited using the techniques introduced here. The techniques would
apply to algorithms in real algebraic geometry where transversality or
Noether position are required geometric properties established by a
random change of coordinates.

\paragraph*{Outline.} The next section summarizes the main concepts from 
algebraic geometry needed in this paper. In Section~\ref{ssec:detvar},
we compare the descriptions of determinantal varieties by the
vanishing of matrix minors, and through the use of Lagrange
multipliers; these results, while rather simple, are used throughout.
A first application is in Section~\ref{sec:wt}, where we give a
quantitative form of Thom's ``weak transversality lemma''.

Section~\ref{sec:overview} introduces polar varieties and discusses the
algorithm sketched above and the genericity conditions required for it
to succeed. These conditions are studied in detail in
Sections~\ref{sec:applications},~\ref{ssec:Hi2} and~\ref{Sec:Hip};
this allows us to complete the analysis of the algorithm in
Section~\ref{sec:analysis}, thereby proving Theorem~\ref{theo:main}.


\section{Preliminaries}

In this section, we gather several basic definitions and properties of
algebraic sets and locally closed sets. General references for this
material are~\cite{Mumford76,Shafarevich77,ECA}.

\paragraph*{Algebraic sets.} An algebraic set $V \subset \C^n$ is the set of common 
zeros of an ideal $I$ in $\C[X_1,\hdots,X_n].$ Conversely, the ideal
of a subset $V$ of $\C^n$, that is, the set of polynomials in
$\C[X_1,\dots,X_n]$ that vanish at all points of $V$, is called the
\textit{ideal} of $V$; this is a radical ideal, which we write $I(V)$.

The smallest algebraic set containing an arbitrary set $Y$ is called 
the {\em Zariski closure} of $Y$ and written $\overline Y$.

\paragraph*{Irreducible decomposition.}
An algebraic set $V \subset \C^n$ is \textit{irreducible} when $V =
V_1 \cup V_2$, with $V_1,V_2$ algebraic sets, implies $V=V_1$ or
$V=V_2$; this is the case if and only if $I(V)$ is prime.  An
algebraic set $V \subset \C^n$ can be decomposed into a finite union
of irreducible algebraic sets
\[
V = V_1 \cup V_2 \cup \cdots \cup V_r,
\]
with $V_i \not \subset V_j$ for all $i \ne j$. The sets
$V_1,\dots,V_r$ are called the {\em irreducible components} of $V$;
they are uniquely defined, up to order. In terms of ideals, $I(V)$
being radical, it admits a decomposition as an irredundant intersection
of prime ideals $I_1,\dots,I_r$; the irreducible algebraic sets
$V(I_1),\dots,V(I_r)$ are the irreducible components of $V$.

\paragraph*{Dimension.}
The \textit{dimension} of an algebraic set $V \subset \C^n$, denoted
$\dim(V)$, can be defined as the unique integer $d$ such that $V \cap
H_1 \cap \cdots \cap H_d$ is finite, but not empty, for a generic
choice of hyperplanes $H_1,\dots,H_d$. The \textit{codimension} of $V$
is $n - \dim (V)$. 

An algebraic set $V$ is \textit{equidimensional} if each of its
irreducible components has the same dimension; if each component has
dimension $d$ then we say that $V$ is $d$-equidimensional.

\paragraph*{Degree.}
We use the definition of degree from~\cite{H}: the \textit{degree}
$\deg(V)$ of an irreducible algebraic set $V$ is the number of
intersection points between itself and $\dim (V)$ generic hyperplanes,
and the degree of an arbitrary algebraic set is defined as the sum of
the degrees of its irreducible components.

The degree of a hypersurface defined by a squarefree polynomial $f$ is
$\deg(f)$. We particularly care about algebraic sets of dimension zero; by
definition, these sets are finite and their degree is equal to their
cardinality.

We will often apply the B\'ezout bound from \cite[Theorem 1]{H},
which says that $\deg(V \cap V') \le \deg(V) \deg(V')$ holds for all
algebraic sets $V,V'$. A last useful property is that for any linear
mapping $\psi: \C^n \to \C^m$, $\deg(\overline{ \psi(V)}) \le \deg(V)$.

\paragraph*{Noether position.}
Suppose that the ambient dimension $n$ is fixed.  For $i$ in
$\{1,\dots,n\},$ let $\pi_i$ denote the projection
\begin{align*}
\C^n  &\rightarrow \C^i \\
(x_1,\hdots,x_n) &\mapsto  (x_1,\hdots,x_i).    
\end{align*} 
A $d$-equidimensional algebraic set $V \subset \C^n$ is in
\textit{Noether position} for the projection $\pi_d$ when the
extension \[\C[X_1,\hdots,X_{d}] \rightarrow \C[X_1,\hdots,X_n]/I(V)\]
is injective and integral; here, $I(V) \subset \C[X_1,\hdots,X_n]$ is
the defining ideal of $V$. It is then a consequence that for any $\xb$ in
$\C^d,$ the fiber $V \cap \pi_d^{-1}(\xb)$ has dimension zero and is
thus finite and not empty.

\paragraph*{Gradient vectors and Jacobian matrices.} The gradient vector 
of a polynomial $f \in \C[X_1,\dots,X_n]$ is written $\grad(f) \in
\C[X_1,\dots,X_n]^{1 \times n}$ (so this is a row vector). Most of the
time, the variables with respect to which we differentiate are clear
from the context, but we may write $\grad_{\bm X}(f)$ for clarity,
with $\bm X=X_1,\dots,X_n$.

The Jacobian matrix of polynomials $F=f_1,\dots,f_s$ is the $s \times
n$ matrix $\jac(F)$, with $\partial f_i/\partial X_j$ at entry
$(i,j)$, for $1\le i \le s$ and $1 \le j \le n$. As we do for
gradients, we will write $\jac_{\bm X}(F)$ if we want to highlight
what variables we differentiate with respect to.

Given $\xb$ in $\C^n$, we then write $\grad(f,\xb)$, resp.\ $\jac(F,
\xb)$, for the evaluation of respectively $\grad(f)$ and $\jac(F)$ 
at $\xb$.

\paragraph*{Tangent spaces, regular and singular points.}
Assume that $V \subset \C^n$ is a $d$-equidimen\-sional algebraic
set. The \textit{Zariski-tangent space} to $V$ at $\xb \in V$ is the
vector space $T_{\xb}V \subset \C^n$ defined by the equations
\[
\grad (g,\xb) \cdot \bm v = 0 \text{~for all $g \in I(V)$}, \quad \bm
v \in \C^{n \times 1}.
\] 
Then, the point $\xb \in V$ is a \textit{regular point} (or
non-singular) if $\dim (T_{\xb}V) = d$; otherwise, $\xb$ is a
\textit{singular point}. We let $\reg(V)$ and $\sing(V)$ respectively
denote the regular and singular points of $V$; when the latter is
empty, we say that $V$ is \textit{smooth}. If $I(V)$ is generated by
polynomials $G=(g_1,\hdots,g_s) \in \C[X_1,\hdots,X_n]^s$, then at any
point $\xb$ of $\reg(V)$, the Jacobian matrix $\jac(G,\xb)$ has rank
$n - d$ and the right kernel of $\jac(G,\xb)$ is $T_{\xb}V.$

\paragraph*{Changes of variables.}
For a matrix $\mA$ in $\C^{n\times n}$ and a polynomial $g$ in
$\C[X_1,\hdots,X_n],$ we write \[g^\mA:=g(\mA \Xb) \in
\C[X_1,\dots,X_n],\] where $\Xb$ is the column vector with entries
$X_1,\dots,X_n$. Similarly, for a sequence of polynomials
$G=(g_1,\hdots,g_s)$ in $\C[X_1,\hdots,X_n]^s$, we write $G^{\mA} =
\left(g_1^{\mA},\hdots,g_s^{\mA}\right).$ For an algebraic set $V
\subset \C^n$ and a matrix $\mA \in \GL(n),$ we define $V^{\mA}$ as
the image of $V$ by the map $\phi_{\mA} : \xb \mapsto \mA^{-1}\xb.$
 Notice in particular that
$V(G^{\mA}) = \phi_{\mA}(V(G)) = V(G)^{\mA}. $

\paragraph*{Locally closed sets.}
We will also need to work with {\em locally closed} sets: we say that
$Y \subset \C^n$ is locally closed if we can write it as $Y=V-V'$, for
some algebraic sets $V,V'$.

The notions of dimension and equidimensionality carry over to this
context (they are defined through the Zariski closure of $Y$), as does
that of tangent space: for $\xb$ in $Y$, we set $T_\xb Y = T_\xb V$
(this is independent of the choice of $V,V'$ in the definition above).
If $Y$ is equidimensional, as we did for algebraic sets, we can then
define the {\em regular points} (or non-singular points) of $Y$ as
those points at which the tangent space has dimension $d$, and we say
that $Y$ is smooth if all its points are regular.

Open sets are locally closed. As another example, for any
$d$-equidimensional algebraic set $V$, $\reg(V)$ is a smooth
$d$-equidimensional locally closed set.


\section{Describing determinantal varieties}\label{ssec:detvar}

In this section, we work with polynomials in $\C[Y_1,\dots,Y_N]$, for
some positive integer $N$. Given a matrix $\bm A$ in
$\C[Y_1,\dots,Y_N]^{q \times r}$, with $q \le r$, together with some
equations $B=(b_1,\dots,b_s)$ in $\C[Y_1,\dots,Y_N]$, we consider the
locus $\X$ defined as 
\[\X = \{\yb \in \C^N \ \mid \ b_1(\yb) = \cdots = b_s(\yb) =0 
     \text{~and~} {\rm rank}(\bm A(\yb)) < q\}.\] One of our goals
     here is to give a degree bound for $\X$; this will be used twice,
     in the next section for our discussion of the weak transversality
     lemma (in a slightly more general context where we work in an
     open subset of $\C^N$), then also to control the degrees of the
     systems of equations we will solve.

Consider the polynomials
\[\frkJ(\bm A, B) = (b_1,\dots,b_s,M_1,\dots,M_P),\]
where $M_1,\dots,M_P$ are the $q$-minors of $\bm A$, with $P={r \choose
q}$. Since we have $V(\frkJ(\bm A, B)) = \X$, we may derive a degree
bound on $\X$ using the B\'ezout inequality. However, even the refined
form given in \cite[Proposition 2.3]{Heintz1980} involves an
exponential dependency in either the ambient dimension $N$ or the
number of minors $P$. This might be acceptable in some contexts (such
as when estimating the degrees of polar varieties), but is way beyond
our target bound in the context of weak transversality, for instance.

Instead, we use Lagrange systems. We let $L_1,\hdots,L_q$ be new
variables, thought of as Lagrange multipliers, and consider the
``Lagrange polynomials'' given as the $r$ entries of $ [ L_1 ~\cdots~
  L_q]\cdot \bm A$. We denote by $Z \subset \C^{N+q}$ the algebraic set
defined by the vanishing of
\[ ( b_1,\dots,b_s, \  [ L_1 ~\cdots~ L_q]\cdot \bm A )\]
and by $Z'$ the algebraic set
\[
Z' := \overline{Z - \{(\yb,0,\dots,0) \in \C^{N+q}~|~(\yb,0,\dots,0) \in Z\}},
\]
where the bar denotes Zariski closure (we have to remove such points,
since $L_1=\cdots=L_q=0$ is always a trivial solution to the Lagrange
equations). Finally, consider the projection
\begin{align*} 
  \mu :~ \C^{N+q} &\rightarrow \C^{N}\\
  (\yb,\bm \ell)~ &\mapsto \yb.
\end{align*}
It is then possible to prove that $\X$ is the Zariski closure of
${\mu(Z')}$, and derive degree bounds using the equations defining
$Z$. However, while introducing $Z'$ is convenient, computing defining
equations for it is non-trivial, as it involves saturation; besides,
in several contexts, it will be advantageous to work with equations in
complete intersection, which the following construction will guarantee  in
certain cases. For $\bm u = (u_1,\dots,u_q) \in \C^q$, consider the
equations
\[\frkL(\bm A, B,\bm u) = ( b_1,\dots,b_s, \  [ L_1 ~\cdots~ L_q]\cdot \bm A,\ u_1 L_1 + \cdots + u_q L_q -1 ),\]
and let $Z_{\bm u} \subset \C^{N+q}$ be its zero-set. Using the linear
equation $u_1 L_1 + \cdots + u_q L_q -1$ allows us to discard
solutions where $L_1 = \cdots =L_q = 0$, but unlucky choices of $\bm
u$ may discard other components as well. The following proposition
makes this more precise.

\begin{prop}\label{prop:projection}
  For any $\bm u$ in $\C^q$, we have the inclusion $\mu(Z_{\bm u})
  \subset \X$. There exists a non-empty open set $\mathscr{O} \subset
  \C^q$ such that for $\bm u$ in $\mathscr{O}$, we have the inclusion
  $\X \subset \overline{\mu(Z_{\bm u})}$, and thus the equalities $\X =
  \overline{\mu(Z_{\bm u})}$ and 
  \[\sqrt{\langle \frkL(\bm A, B,\bm u)\rangle\cap \C[Y_1,\dots,Y_N]} = 
  \sqrt{\langle \frkJ(\bm A, B) \rangle}.\]
  The set $\mathscr{O}$ is the complement of at most $\deg(\X)$ 
  hyperplanes.
\end{prop}
\begin{proof}
  If $(\bm y, \bm \ell)$ cancels all polynomials in ${\frkL(\bm A,
    B,\bm u)}$, then $\bm \ell$ is non-zero, so that $\bm A(\bm y)$ is
  rank-deficient. As a consequence, $\bm y$ is in $\X$. This proves
  the first assertion.

  For the second one, let $\X_1,\dots,\X_K$ be the irreducible
  components of $\X$.  For any given $k$ in $\{1,\dots,K\}$, since all
  $q$-minors of $\mA$ vanish on $\X_k$, they vanish in the function
  field $\C(\X_k)$, so $\mA$ has rank less than $q$ as a matrix over
  $\C(\X_k)$. Thus, there exists a non-zero vector of rational
  functions
  \[\bm \ell_k = (\ell_{k,1},\hdots,\ell_{k,q})=\left(\frac{N_{k,1}}{D_k},\hdots,\frac{N_{k,q}}{D_k}\right)\in \C(\X_k)^{ q},\]
  such that $\bm \ell_k  \cdot \bm A = 0$ in $\C(\X_k)^r$ 
(here, we see $\bm \ell_k$ in $\C(\X_k)^{1 \times q}$).
  For definiteness, assume that $N_{k,\iota_k} \ne 0.$ Then, in particular,
  $\X'_k = \X_k - V( D_k N_{k,\iota_k})$ is dense in $\X_k$; for 
  $\bm y$ in $\X'_k$, $\bm\ell_k(\bm y)$ is well-defined, non-zero, and still
  satisfies $\bm \ell_k(\bm y) \cdot \bm A(\bm y) = 0$.
  
  Then, pick a point $\bm y_k$ in $\X'_k$, so that
  $\bm \ell_k(\bm y_k)$ is a well-defined, non-zero vector in
  $\C^q$. This allows us to define a non-empty Zariski open set
  $\mathscr{O}_k \subset \C^q$ by the condition
  \[
  \mathscr{O}_k := 
  \left\{\bm u \in \C^q~|~ \bm \ell_k(\bm y_k) \cdot \bm u \ne 0\right \},
  \]
  where in the dot product we take $\bm \ell_k(\bm y_k)$ in $\C^{1 \times q}$ and $\bm u$ in
  $\C^{q \times 1}$.  Finally, we let $\mathscr{O} := \cap_{1 \le k
    \le K} \mathscr{O}_{k}$, which is open, non-empty, and defined 
  as the complement of $K \le \deg(\X)$ hyperplanes, as claimed. We now prove
  that for $\bm u$ in $\mathscr{O}$, the inclusion $\X \subset
  \overline{\mu(Z_{\bm u})}$ holds.

  For this, we take $k$ as above, 
  and we prove that $\X_k$ is contained in $\overline{\mu(Z_{\bm u})}$.
  Consider the rational mapping 
  \begin{align*}
    \X'_k  &\rightarrow \C\\    
    \bm y &\mapsto  \bm \ell_k(\bm y)\cdot \bm u = \frac{ u_1 N_{k,1}(\bm y) + \cdots + u_q N_{k,q}(\bm y)}{D_k(\bm y)}.    
  \end{align*}
  Put $\X''_k = \X'_k - V(u_1 N_{k,1} + \cdots + u_q N_{k,q})$; this is
  again an open subset of $\X_k$, and the fact that $ \bm
  \ell_k(\bm y_k)\cdot \bm u$ is non-zero, with $\bm y_k$ in $\X'_k$, shows that
  $\X''_k$ is not empty. In particular, it is dense in $\X_k$.   

  Take $\bm y$ in $\X''_k$. Then, $\alpha:=\bm \ell_k(\bm y)\cdot \bm u$
  is non-zero, set we can define $\bm \ell' := 1/\alpha\ \bm\ell_k(\bm y)$.
  Then, $\bm \ell'$ is still in the left nullspace of $\bm A(\bm y)$,
  and by construction $ \bm \ell'\cdot \bm u =1$, so that $(\bm y, \bm
  \ell')$ is in $Z_{\bm u}$. In other words, $\X''_k$ is contained in
  $\mu(Z_{\bm u})$. Taking the Zariski closure, we obtain that $\X_k$ 
  is contained in $\overline{\mu(Z_{\bm u})}$, as claimed. 
  The equality $\X = \overline{\mu(Z_{\bm u})}$ follows, as does the
  claimed equality between ideals.
\end{proof}

\begin{corollary} \label{coro:degree}
  If all polynomials $b_1,\dots,b_s$ have respective degrees at most
  $d_1,\dots,d_s$, and all entries of $\bm A$ have degree at most
  $d'$, then the degree of $\X$ is at most $d_1 \cdots d_s (d'+1)^r.$
\end{corollary}
\begin{proof}
  Choose $\bm u$ in the set $\mathscr{O}$ of the previous lemma.  The
  algebraic set $Z_{\bm u}$ is defined by $s$ equations of respective
  degrees at most $d_1,\dots,d_s$, $r$ equations of degree at most
  $d'+1$ and a linear equation. It follows from B\'ezout's
  Theorem~\cite{H} that $\deg(Z_{\bm u}) \leq d_1 \cdots d_s
  (d'+1)^r$. Degree does not increase through projection, so the
  conclusion follows from the previous lemma.
\end{proof}

\begin{remark}\label{rk:degree}
  In the next section, we will consider the following slight variant
  of the problem considered here, where we are interested in the
  locally closed set
  \[\X' = \{\yb \in \Omega \ \mid \ b_1(\yb) = \cdots = b_s(\yb) =0 
  \text{~and~} {\rm rank}(\bm A(\yb)) < q\},\] for some Zariski open
  set $\Omega \subset \C^N$.  The Zariski closure $\overline{\X'}$ is
  the union of certain irreducible components of the set $\X$ defined
  above, so the degree bound of Corollary~\ref{coro:degree} still
  holds for $\overline{\X'}$.
\end{remark}

We note that in some cases, sharper bounds are known for the degrees
of determinantal varieties: for instance, when $X$ is finite, and
defined as the set of critical points on a smooth algebraic
set~\cite{SaSp16}, or when we want to take into account differences in
the degrees of the rows and columns of $\bm
A$~\cite{Spa14,NieRan09,HaSaScVu18}.


\section{Weak transversality}\label{sec:wt}

Several of the generic properties of polar varieties are consequences
of {\em weak transversality}, which is an important extension of
Sard's lemma due to Thom (this observation goes back to work of
Giusti, Heintz and collaborators~\cite{BaGiHeMb97,BaGiHeLePa12}).  In
this section, we develop a quantitative extension of Thom's weak
transversality theorem, specialized to the particular case of
transversality to a point. In the sequel, we will apply this result to
bound the degree of particular hypersurfaces our algorithm needs to
avoid to guarantee success.


\subsection{Definitions and statement of the result}

In its differential version, Sard's lemma states that the set of
critical values of a smooth function $\R^n \to \R^m$ has measure zero;
extensions exist to smooth mapping between differential manifolds.  In
our algebraic context, we will use the following definitions.

Consider a polynomial mapping $\Phi : Y \rightarrow \C^m$ from a
smooth $n$-equidimensional locally closed set $Y$ to $\C^m$, with
$m\le n$. A {\em critical point} of $\Phi$ is a point $\bm y \in Y$
for which the image of the tangent space $T_{\bm y} Y$ by the Jacobian
matrix $\jac(\Phi,\bm y)$ has dimension less than $m$. For instance,
the case that will interest us in this section is when $Y$ is Zariski
open in $\C^n$, in which case we have $T_{\bm y} Y=\C^n$ for all $\bm
y$ in $Y$, and the condition is equivalent to the Jacobian of $\Phi$
having rank less than $m$ at $\bm y$. {\em Critical values} are the
images by $\Phi$ of critical points; the complement of this set are
the {\em regular values}. Notice then, a regular value is not
necessarily in the image of~$\Phi$.

One can then give ``algebraic'' versions of Sard's lemma: for
instance,~\cite[(3.7)]{Mumford76} shows that for $Y$ an irreducible
algebraic set and $\Phi$ dominant, the critical values of $\Phi$ are
contained in a strict algebraic subset of $\C^m$; below, we will rely
on a straightforward generalization given in~\cite{TWT}. See
also~\cite[Chapter~9]{bochnak1998real} for the semi-algebraic case.

Thom's weak transversality lemma, as given for instance
in~\cite{demazure2000bifurcations}, generalizes Sard's lemma. In this
section, we consider a particular case of this result (transversality
to a point), and establish a quantitative version of it.

Let $n,\dt,$ and $m$ be positive integers, with $m \le n$ as before,
let $\mathscr{O}$ be a Zariski open subset of $\C^n$, and denote by
$\Phi: \mathscr{O} \times \C^{\dt} ~ \rightarrow \C^{m}$ a mapping
given by polynomials in $n+\dt$ indeterminates
$X_1,\dots,X_n,\Theta_1,\dots,\Theta_\dt$ (the latter should be
thought of as parameters). For $\thetab$ in $\C^{\dt}$, we let
$\Phi_{\thetab} : \mathscr{O} \rightarrow \C^{m}$ be the induced
mapping $\xb\mapsto\Phi(\xb,\thetab)$.  Thom's weak transversality
lemma tells us that if $0$ is a regular value of the mapping $\Phi$,
then $0$ remains a regular value of the induced mapping $\Phi_{\bm
  \vt}$ for a generic $\bm \vt$. (Here, we are dealing with the
particular case of transversality to a point, which can be rephrased
entirely in terms of regular and critical values.)  Our quantitative
version of this result is the following.

\begin{prop} [Weak transversality]\label{prop:weak_t}
  Let all notation be as before, and suppose that $\Phi$ is defined by
  $m$ polynomials of degree at most $d$. If $0$ is a regular value of
  $\Phi$, there exists a non-zero polynomial $\Gamma \in
  \C[\Theta_1,\dots,\Theta_s]$ of degree at most $d^{m+n}$ such that
  for $\thetab$ in $\C^\dt$, if $\Gamma(\thetab)\ne 0$, then $0$ is a
  regular value of~$\Phi_{\thetab}$.
\end{prop}
\begin{ex}
  Consider a squarefree polynomial $f$ in $\C[X_1,X_2]$, with degree
  at most $d$, defining a smooth curve $V(f)$ in $\C^2$, and let the
  mapping $\Phi:\C^2\times \C \to \C^2$ be defined by
  $\Phi(X_1,X_2,\Theta) = (f(X_1,X_2), X_1-\Theta)$ (so $m=n=2$ and
  $s=1$). One checks that the Jacobian of $\Phi$ with respect to
  $(X_1,X_2,\Theta)$ has full rank two at any point in $\Phi^{-1}(0)$,
  so that $0$ is a regular value of $ \Phi$ and therefore the
  assumptions of the proposition apply.

  We then deduce that a non-zero polynomial $\Gamma \in \C[\Theta]$
  exists, with degree at most $d^{4}$ with the property that, if
  $\vartheta$ in $\C$ does not cancel $\Gamma$ then $0$ is a regular
  value of the induced mapping $ \Phi_{\vartheta}$. In particular, for
  all $\vartheta$ in $\C$ except at most $d^4$ values, the ideal
  $(f(X_1,X_2), X_1-\vartheta)$ is radical in $\C[X_1,X_2]$;
  equivalently, $f(\vartheta, X_2)$ is squarefree.
\end{ex}
In this example, we could of course obtain the same result (with a
sharper degree bound) by considering the discriminant of $f$ with
respect to $X_2$, but the construction above will be useful later on,
in a generalized form. (In this example, the bound $d^4$ could be
sharpened by utilizing the fact that only one of the polynomials
defining $\Phi$ has degree $d$, whereas the other one is linear.)

The rest of the section is devoted to the proof of the proposition.
The proof of \cite[Theorem B.3]{TWT} already shows the existence of
$\Gamma$; it is essentially the classical proof for smooth
mappings~\cite[Section~3.7]{demazure2000bifurcations}, written in an
algebraic context. In what follows, we revisit this proof,
establishing a bound on the degree of $\Gamma$.


\subsection{Proof of the proposition}

In what follows, we use the notation of Proposition~\ref{prop:weak_t},
so that we consider $m$ polynomials $\Phi$ that depend on variables
$X_1,\dots,X_n$ and $\Theta_1,\dots,\Theta_s$, with $m \le n$, and an
open set $\sO \subset \C^n$.

In the context of Thom's weak transversality, the ``bad'' parameter
values show up as the critical values of a certain projection. 
Put $Y = \Phi^{-1}(0) \cap (\sO \times \C^s)$, and let $V$ be the
Zariski closure of $Y.$ If $Y$ is empty, there is nothing to do, since
all values $\thetab$ in $\C^\dt$ satisfy the conclusion of the
proposition. We therefore assume that $Y$ is not empty. Take $(\xb,
\thetab)$ in $Y$; then by assumption, $\jac({\Phi},(\xb,\thetab))$ has
full rank $m$. Since in a neighbourhood of $(\xb,\thetab)$, $V$
coincides with $Y={\Phi}^{-1}(0) \cap (\sO \times \C^s)$, the Jacobian
criterion~\cite[Corollary 16.20]{ECA} implies that there is a unique
irreducible component $V_{(\xb,\thetab)}$ of $V$ that contains
$(\xb,\thetab),$ that $(\xb,\thetab)$ is regular on this component and
that $\dim V_{(\xb,\thetab)}=n+s-m$. This implies that $Y$ is a
smooth, $(n+s-m)$-equidimensional locally closed set.  Now, consider
the projection
\begin{align*}
  \pi: \C^{n+\dt} &\rightarrow\C^{\dt} \\
  (\xb, \thetab)&\mapsto\thetab, 
\end{align*}
and let $Z$ be the set of critical points of  the restriction $\pi_{|Y}$ of $\pi$ to $Y$; that is, \[Z
:= \{(\xb,\bm \thetab) \in Y~|~\dim (\pi(T_{\xb,\bm \thetab}Y))<s\}.\]
The projection $\pi(Z) \subset \C^s$ is thus the set of critical values
of $\pi_{|Y}$.

\begin{lemma}
  The Zariski closure $\overline{\pi(Z^{})}$ is a strict subset of
  $\C^s$.
\end{lemma}
\begin{proof}
  The discussion above implies that $\reg(V)$ is a smooth,
  $(n+s-m)$-equidimensional locally closed set containing $Y$.  Let
  then $Z^{'}$ be the critical points of $\pi_{|\reg(V)}$; by the
  algebraic form of Sard's lemma as given in
  \cite[Theorem~3.7]{Mumford76} (for irreducible $V$)
  and~\cite[Proposition~B.2]{TWT} (for general $V$), the Zariski
  closure $\overline{\pi(Z^{'})}$ is a strict closed subset of
  $\C^s$. Now, at any point $(\xb,\thetab)$ of $Y$, the tangent spaces
  $T_{(\xb,\thetab)} Y$ and $T_{(\xb,\thetab)} \reg(V)$ coincide. As a
  result, $Z$ is contained in $Z^{'}$, and the claim follows.
\end{proof}

\noindent 
We can now explain how $\thetab$ being a regular value of $\pi_{|Y}$
relates to $0$ being a regular value of $\Phi_{\thetab}$.  In what
follows, we write our indeterminates as blocks of variables, with
$\Xb=X_1,\dots,X_n$ and $\Thetab = \Theta_1,\dots,\Theta_s$. When 
not explicitly mentioned, Jacobian matrices involve derivatives
with respect to both $\bm X$ and $\bm \Theta$.

\begin{lemma}\label{prop:rankJ}
  For $(\xb,\thetab)$ in $Y$, $(\xb,\thetab)$ is in $Z$ if and only if
  $\jac_{\bm X}(\Phi,(\xb,\thetab))$ has rank less than $m$.
\end{lemma}
\begin{proof}
  Let $\bm M$ denote the $(s+m) \times (s+n)$ Jacobian matrix of $\pi$
  and $\Phi$ with respect to $X_1,\dots,X_n$ and
  $\Theta_1,\dots,\Theta_s$, that is,
  \begin{align*}
    \bm M &= 
    \bbm 
    \jac(\pi)\\
    \jac(\Phi) 
    \ebm 
    =
    \bbm 
    \textbf{0}_{\dt \times n}\hspace{5mm}\textbf{I}_{\dt} \\
    \jac(\Phi)
    \ebm.
  \end{align*}
  Take $(\xb,\thetab)$ on $Y$. Then, the rank of $\bm M(\xb,\thetab)$
  can be written as $\textup{rank}(\jac(\Phi,(\xb,\thetab))) +
  \textup{rank}([\textbf{0}_{\dt \times n}~\textbf{I}_{\dt}] \mid \ker
   \jac(\Phi,(\xb,\thetab)))$, where the latter is the rank of the
  restriction of $[\textbf{0}_{\dt \times n}~\textbf{I}_{\dt}]$ to the
  nullspace of $\jac(\Phi,(\xb,\thetab))$.

  Since $(\xb,\thetab)$ is in $Y$ and since $0$ is a regular value of
  $\Phi$, $\jac (\Phi,(\xb,\thetab))$ has full rank $m$. On the other
  hand, the nullspace of that matrix is the tangent space
  $T_{\xb,\thetab} Y$, and $\textup{rank}([\textbf{0}_{\dt \times
      n}~\textbf{I}_{\dt}] \mid \ker \jac(\Phi,(\xb,\thetab)))$ is the
  dimension of $\pi(T_{\xb,\thetab} Y)$.  In other words, the rank of
  $\bm M(\xb,\thetab)$ is equal to $m+\dim(\pi(T_{\xb,\thetab} Y))$.

  This proves that for $(\xb,\thetab)$ in $Y$, $(\xb,\thetab)$ is in
  $Z$ if and only if the matrix $\bm M$ has rank less than $\dt+m$ at
  $(\xb,\thetab)$. Now, notice that
  \begin{align*}
   \bm M(\xb,\thetab)&= 
    \bbm 
    \textbf{0}_{\dt \times n} &\textbf{I}_{\dt} \\
     \jac_{\bm X}(\Phi,(\xb,\thetab))     &\jac_{\bm \Theta}(\Phi, (\xb,\thetab))
    \ebm.
  \end{align*}
  This shows that the rank of $\bm M(\xb,\thetab)$ 
  equals $s + \textup{rank}(\jac_{\bm X}(\Phi,(\xb,\thetab)))$,
  and  the lemma follows.
\end{proof}

As a result, suppose we take $\thetab$ in $\C^\dt - {\pi(Z)}$.  Then
for all $\xb$ in $\Phi_{\thetab}^{-1}(0) \cap \sO$, $(\xb,\thetab)$ is
in $Y$, so it is not in $Z$; the previous lemma then implies that the
Jacobian matrix of $\Phi_{\thetab}$, which is $\jac_{\bm X}(\Phi,(\bm
X,\thetab))$, has full rank $m$ at $\xb$. In other words, $0$ is a
regular value of $\Phi_{\thetab}$ in the open set $\sO$. To prove
Proposition~\ref{prop:weak_t}, it is thus enough to establish the
existence of a non-zero polynomial of degree at most $d^{m+n}$ that
vanishes on $\overline{\pi(Z)}$. We already established that
$\overline{\pi(Z)}$ is a strict subset of $\C^\dt$, so the only
missing ingredient is to prove that it has degree at most $d^{m+n}$.

We start by bounding above the degree of $\overline{Z}$.
The previous lemma shows the equality
\[Z = \{ (\xb,\thetab) \in \Omega \times \C^\dt \ \mid \Phi(\xb,\thetab) =0
\text{~and~} {\rm rank}(\jac_{\bm X}(\Phi,(\xb,\thetab))) < m\}.\]
Since all polynomials in $\Phi$ have degree at most $d$, and all
entries of $\jac_{\bm X}(\Phi)$ at most $d-1$, we can apply
Corollary~\ref{coro:degree}, so as to deduce that $\deg(\overline{Z})
\le d^{m+n}$. This implies that $\overline{\pi(\overline Z)}$ has
degree at most $d^{m+n}$, and the equality $\overline{\pi(\overline
  Z)} =\overline{\pi(Z)}$ allows us to conclude the proof.


\section{Overview of the main algorithm}\label{sec:overview}

Let $F = (f_1,\hdots,f_p)$ be a sequence of polynomials in
$\C[X_1,\hdots,X_n]$. Suppose that the ideal $\langle F \rangle
\subset \C[X_1,\hdots,X_n]$ is radical and that $V(F)$ is smooth and
equidimensional of dimension $\delta= n-p$.

In this section, we give a high-level description of an algorithm
from~\cite{EMP} that computes at least one point in each connected
component of $V(F) \cap \R^n$. Correctness of this algorithm was
established in~\cite{EMP} provided we are in generic coordinates: the
algorithm solves a family of systems of equations that describe points
on the {\em polar varieties} of $V(F)$, and being in generic
coordinates ensures several desirable properties for these polar
varieties.

After a brief review of the basic properties of polar varieties, we
sketch the main algorithm and highlight what properties are needed for
its correctness (the next sections will give quantitative statements
regarding the genericity of these properties). In that, we mainly
follow~\cite{EMP}, but we also introduce requirements related to
Lagrange systems, as introduced in Section~\ref{ssec:detvar}, as they
will be of help in further sections.


\subsection{Polar varieties}

Let $F$ be as in the preamble and let $V=V(F)$. Recall that, for $i
\in \{1,\hdots,n\},$ we denote by $\pi_i$ the projection
\begin{align*}
  \C^n &\rightarrow \C^i \\
  (x_1,\hdots,x_n) &\mapsto  (x_1,\hdots,x_i).    
\end{align*} 
For $i \le \delta$, the $i$-th \textit{polar variety} $W(i,F)$  is
the set of critical points of the restriction of $\pi_i$ to $V$, that
is,
\[W(i,F) := \left\{\xb \in V~|~\dim \pi_i(T_\xb V) < i\right\}.\]
We naturally extend this definition to $i=\delta+1$, by setting
$W({\delta+1},F)=V$.

For $1\le i \le \delta+1$, let $\jac(F)$, resp.\ $\jac(F, i)$, denote the
Jacobian matrix of $F=(f_1,\hdots,f_p)$ with respect to
$(X_{1},\hdots,X_n)$, resp.\ to $(X_{i+1},\hdots,X_n):$
\[
\jac(F)=
\left[ 
\begin{array}{ccc}
\frac{\pa f_1}{\pa X_{1}}&\hdots& \frac{\pa f_1}{\pa X_{n}} \\
\vdots& &\vdots\\
\frac{\pa f_p}{\pa X_{1}}&\hdots& \frac{\pa f_p}{\pa X_{n}} 
\end{array}
\right ],
\quad
\jac(F, i)=
\left[ 
\begin{array}{ccc}
\frac{\pa f_1}{\pa X_{i+1}}&\hdots& \frac{\pa f_1}{\pa X_{n}} \\
\vdots& &\vdots\\
\frac{\pa f_p}{\pa X_{i+1}}&\hdots& \frac{\pa f_p}{\pa X_{n}} 
\end{array}
\right]. 
\]
Since $F$ generates a radical ideal, for any $\xb$ in $V$, the tangent
space $T_\xb(V)$ is the kernel of $\jac(F,\xb)$; the assumption that
$V$ be $\delta$-equidimensional and smooth implies that this kernel
has dimension $\delta=n-p$ at all such $\xb$. It follows that 
we can rephrase the definition of $W(i,F)$ as
\[W(i,F) = \left\{\xb \in \C^n~|~ f_1(\xb)=\cdots=f_p(\xb)=0 
\text{~and~} {\rm rank}(\jac(F,i,\xb)) < p\right\}.\] Let $P_i
=\binom{n-i}{p}$ be the number of $p$-minors in $\jac(F,i)$, and let
$M_{i,1},\hdots,M_{i,P_i}$ be these minors (for $i=\delta+1$,
$P_{\delta+1}=0$ since $\jac(F,{\delta+1})$ has size $p \times
(p-1)$). Then, as in Section~\ref{ssec:detvar}, we deduce that
$W(i,F)$ is defined by the polynomials
\begin{equation}\label{eq:frkJ}
\frkJ(i,F) = \big(f_1,\hdots,f_p,M_{i,1},\hdots,M_{i,P_{i}}\big).  
\end{equation}
The downside to defining polar varieties using minors of the truncated
Jacobian matrix is that these equations are in general not complete
intersection, due to the relations between minors of a matrix (the
hypersurface case is an exception, since in this case only partial
derivatives are used to define polar varieties). For both the
polynomial system algorithm we will use below, and an application we
will make of an effective Nullstellensatz, it will be necessary to
have equations without such relations. To make this possible, we use
an alternative modeling of polar varieties that uses Lagrange
variables, as in Section~\ref{ssec:detvar}. We may thus consider the zero-set of the polynomials
\[ \big(F,\ [L_1~\cdots~L_p]\cdot \jac(F, i)\big ) \in \C[X_1,\dots,X_n,L_1,\dots,L_p]^{p+n-i},\]
but as before, we will want to discard from the zero-set of these
equations in $\C^{n+p}$ those components where all $L_i$'s vanish
identically. We pointed out that the saturation needed to remove such components
is unlikely to yield convenient sets of generators, so we will again
introduce a single additional equation, of the form $u_1 L_1 + \cdots
+ u_p L_p -1$, for a certain $\bm u = (u_1,\dots,u_p)$ in
$\C^p$. Thus, for such a vector $\ub$, we define the following
polynomials:
\begin{equation}\label{eqdef:Iil}
\frkL(i,F,\bm u)= 
\big (F,\ [L_1~\cdots~L_p]\cdot \jac(F, i),\ u_1 L_1 + \cdots + u_p L_p -1 \big )
\in \C[X_1,\dots,X_n,L_1,\dots,L_p]^{p+n-i+1}.
\end{equation}
Introducing the last equation discards all solutions with $L_1 =
\cdots = L_p =0$, but other components of interest may be removed as
well. However, Proposition~\ref{prop:projection} shows that for a {\em
  generic} vector $\bm u$, the Zariski closure of the projection of
the zero-set of these equations on the $X_1,\dots,X_n$-space is indeed
$W(i,F)$.  In the algorithm, we will use random $u_i$'s; the former
proposition will allow us to quantify bad choices.


\subsection{The algorithm}\label{ssec:algo}

All notation being as before, we can now give the outline of Safey El
Din and Schost's algorithm for computing at least one point in each
connected component of $V(F) \cap\R^n$. To ensure its correctness, we
will need certain genericity assumptions, which will be discussed in
detail in the next sections.

After applying a randomly chosen change of variables $\mA$, we further
choose random $\bm\sigma=(\sigma_1,\dots,\sigma_{\delta})$ in
$\C^{\delta}$, with $\delta = n -p$. Then, for $i=1,\dots,\delta+1$,
we compute (in the new coordinates) the points $\xb=(x_1,\dots,x_n)$
satisfying
\begin{equation}\label{eq:syst}
x_1 = \sigma_1,\hdots,x_{i-1} = \sigma_{i-1},\ f_1(\xb)=\cdots=f_p(\xb)=0,\ {\rm rank} (\jac(F, i,\xb)) < p.
\end{equation}
In geometric terms, this means that we compute the intersection of
$W(i,F)$ with the fiber $\pi_i^{-1}(\sigma_1,\dots,\sigma_{i-1})$.
Then, we return the union of all these sets.

Departing from~\cite{EMP}, and following the discussion in the
previous subsection, we will avoid solving the system generated by
$F=(f_1,\dots,f_p)$ and the $p$-minors of $\jac(F, i)$: to control
costs, it will be beneficial to use the Lagrange system
of~\eqref{eqdef:Iil} instead. Hence, some of our genericity
assumptions will concern these equations. For $i=1,\dots,\delta+1$, we
define the following properties:
\begin{description}
\item [$\bm H_i(1):$] $W(i,F)$ is either empty or $(i-1)$-equidimensional;
\item [$\bm H_i(2):$] $0$ is a regular value of the $n+p-i$ polynomials
  $F,\ [L_1~\cdots~L_p]\cdot \jac(F, i)$ in the open set defined by
  $(L_1,\dots,L_p) \ne (0,\dots,0)$;
\item [$\bm H_i(3):$] assuming $\bm H_i(1)$ holds, $W(i,F)$ is either
  empty or in Noether position for $\pi_{i-1}$.
\end{description}
As we will see, these properties hold after applying a generic change
of variables. Properties $\bm H_i(1)$ and $\bm H_i(3)$ ensure that
Eq.~\eqref{eq:syst} defines a finite set (as a consequence of the
definition of Noether position), and guarantee that the output of the
algorithm contains at least one point in each connected component of
$V(F) \cap \R^n$ (this is proved in~\cite[Theorem~2]{EMP}). The second
one will be used to establish that assumption $\bm H^{'}_i$ defined
below holds generically.

Indeed, assuming (possibly after applying a change of variables) that
$F$ satisfies $\bm H_i$, we define our second genericity property:
\begin{description}
\item [$\bm H^{'}_i:$] $\bm \sigma$ is such that $0$ is a regular value of the $n+p-1$
  polynomials
  \[ X_1 - \sigma_1, \dots, X_{i-1} - \sigma_{i-1},\ F,\ [L_1~\cdots~L_p]\cdot \jac(F, i), \]
  in the open set defined by $(L_1,\dots,L_p) \ne (0,\dots,0)$.
\end{description}
Again, we will see that this property holds for a generic choice of
$\bm\sigma$ and that as a consequence, $0$ is a regular value of the
$n+p$ polynomials
\begin{equation}\label{eq:syst2}
  X_1 - \sigma_1, \dots, X_{i-1} - \sigma_{i-1},\ F,\ [L_1~\cdots~L_p]\cdot \jac(F, i),\ u_1 L_1 + \cdots + u_p L_p -1.
\end{equation}
In particular, these equations admit finitely many solutions.

Suppose that for some $i$ in $\{1,\dots,\delta+1\}$,  $F$ satisfies
$\bm H_i$ and $\bm \sigma$ satisfies $\bm H_i^{'}$; then, we know that
both systems~\eqref{eq:syst} and~\eqref{eq:syst2} have finitely many
solutions. In order to find the solutions of~\eqref{eq:syst}, we will
compute those of~\eqref{eq:syst2} and project them on the
$X_1,\dots,X_n$-space; we choose to solve equations~\eqref{eq:syst2},
since for this input, we can use the algorithm in~\cite{SH}, for which
a complete bit complexity analysis is available. To guarantee success
of this approach, we will rely on our last genericity property:
\begin{description}
\item [$\bm H^{''}_i:$] $\bm u$ is such that the projections of the
  solutions of~\eqref{eq:syst2} on the $X_1,\dots,X_n$-space are the
  solutions of~\eqref{eq:syst1}.
\end{description}
Applying Proposition~\ref{prop:projection} to the polynomials
in~\eqref{eq:syst2} shows that this property holds for a generic
choice of $\bm u$ (notice that since~\eqref{eq:syst2} has finitely
many solutions, taking the Zariski closure, as done in the
proposition, is not necessary in this case). If this is the case, the
previous discussion shows that solving the systems~\eqref{eq:syst2},
for $i=1,\dots,\delta+1$, and projecting their solutions on the
$X_1,\dots,X_n$-space, solves our problem.

The next three sections prove the claims made above on the genericity
of these properties: $\bm H_i(1)$ and $\bm H_i(2)$ in
Section~\ref{sec:applications}, as a first application of weak
transversality; $\bm H_i(3)$ in Section~\ref{ssec:Hi2}, as an
application of an effective Nullstellensatz; and $\bm H'_i$ in
Section~\ref{Sec:Hip}, as another first of weak transversality
(essentially, Sard's lemma). In all cases, we gave quantitative form
of these genericity statements. As we pointed out above,
Proposition~\ref{prop:projection} is enough to prove that $\bm H''_i$
holds for generic $\bm u$, and already gives a quantitative statement.


\section{Genericity of ${\bf H}_i(1)$ and ${\bf H}_i(2)$}
\label{sec:applications}

Notation in this section is as before: we let $F=(f_1,\hdots,f_p) \in
\ZZ[X_1,\hdots,X_n]^p$ be a sequence of polynomials defining a radical
ideal, and where the degree of each polynomial is at most $d$; we 
also assume that the zero-set $V(F) \subset \C^n$ is smooth 
and $\delta$-equidimensional, with $\delta=n-p$. 

Consider an $n \times n$ matrix $\A$ with indeterminates with entries
$(\frak A_{j,k})_{1\le j,k \le n}$. In this section, we prove the
following proposition.

\begin{prop}\label{prop:Hi12}
  For $i=1,\dots,\delta+1$, there exists a non-zero polynomial
  $\D_{i,1}$ in $\C[(\frak A_{j,k})_{1\le j,k \le n}]$ of degree
  at most $n(d^{5n}+1)$ and with the following property. For $\mA$ in
  $\C^{n\times n}$, if $\mA$ does not cancel $\D_{i,1}$, then $F^\mA$
  satisfies ${\bf H}_i(1)$ and ${\bf H}_i(2)$.
\end{prop}
The rest of this section is devoted to the proof of the proposition;
it is based on a construction introduced by Giusti, Heintz {\it et
  al.} (see for instance~\cite{BaGiHeSaSh10}). In all that follows,
$i$ is fixed in $1,\dots,\delta+1$; we then let $\A_{\le i}$ denote
the $in$ indeterminates $(\frak A_{j,k})_{1\le j \le i, 1 \le k \le
  n}$. Writing $\bm X=X_1,\dots,X_n$, we let $\bm K_i(\bm X,\A_{\le
  i})$ denote the $(p+i)\times n$ matrix
\[
\bm
K_i(\bm X, \A_{\le i})=
\bbm 
\jac(F)\\
\A_{1,1}~~ \hdots ~~\A_{1,n}\\
\vdots\hspace{10mm}\vdots\\
\A_{i,1}~~ \hdots ~~\A_{i,n}
\ebm.
\]
Consider elements $\bm a \in \C^{in}$ as vectors of length $i$ of the
form $\bm a = (\bm a_1,\hdots,\bm a_i)$ with $\bm a_i \in \C^n$;
we say that $\bm a$ has rank $i$
when $\bm a$ is a sequence of linearly independent vectors.  Then
for such an $\bm a$, $\bm K_i(\Xb,\bm a)$ is naturally defined with the
indeterminates $\A_{\le i}$ evaluated at $\bm a$. 

Let $\Phi: \C^{n+p+i}\times \C^{i n} \to \C^{n+p}$ be  the polynomial mapping
in indeterminates 
$\bm X=X_1,\dots,X_n$, $\bm L=L_1,\dots,L_p$, $\bm T=T_1,\dots,T_p$ 
and $\A_{\le i}$ 
defined as
\[\Phi = ( F,\ [L_1 ~\cdots~L_p~T_1~\cdots~T_i] \cdot \bm K_i ),\]
and for $\bm a$ in $\C^{ni}$, let $\Phi_{\bm a}: \C^{n+p+i} \to
\C^{n+p}$ be the induced mapping $\Phi_{\bm a} = \Phi(\bm X, \bm L, \bm
T, \bm a)$ in variables $\bm X$, $\bm L$ and $\bm T$.

Let further $\sA \subset \C^{n+p+i}$ be the open set defined by the
condition $(L_1,\dots,L_p) \ne (0,\dots,0)$. In \cite[Section
  3.2]{BaGiHeSaSh10}, it is shown that, for any $(\xb, \bm \lambda,\bm
\thetab, \bm a)$ in $\sA \times \C^{in}$, the Jacobian matrix
$\jac(\Phi)$, taken with respect to all indeterminates $\bm X,\bm
L,\bm T,\A_{\le i}$, has full rank $n+p$ at $(\xb, \bm \lambda, \bm
\thetab, \bm a)$. In particular, this is true for $(\xb, \bm \lambda,
\bm \thetab, \bm a)$ in $\Phi^{-1}(0)$, so that $0$ is a regular value
of $\Phi$ on $\sA \times \C^{in}$.
It therefore follows by Proposition \ref{prop:weak_t} that there
exists a non-zero polynomial $\Gamma_i \in
\C[\A_{1,1},\hdots,\A_{i,n}]$ of degree at most
\[
d^{(n+p+i)+(n+p)} \le d^{5n},
\]
such that if $\ab \in \C^{i n}$ does not cancel $\Gamma_i$,
then $0$ is a regular value of $\Phi_{\bm a}$ on $\sA$. That is, for
$(\xb, \bm \lambda,\bm \thetab) \in \sA \cap \Phi^{-1}_{\bm a}(0)$,
the Jacobian matrix $\jac(\Phi_{\bm a})$ has full rank $n+p$
at ${(\xb, \bm \lambda,\bm \thetab)}$.
   
Let $\fB=\A^{-1}$ in $\C((\frak A_{j,k})_{1\le j,k \le n})^{n \times n}$ and let
$\fB_1=[\fB_{1,1},\hdots,\fB_{1,n}],\hdots,\fB_n=[\fB_{n,1},\hdots,\fB_{n,n}]$
denote the rows of $\fB.$ Set
\[
\Delta_{i,1} := \Gamma_i(\fB_1,\hdots,\fB_i)\cdot (\det( \A))^{\deg (\gi)+1}. 
\]
By multiplying through by $(\det( \A))^{\deg( \gi)+1},$ we cancel all
denominators and make $\D_{i,1}$ a polynomial multiple of $\det(\A)$.
\begin{lemma}
  The degree of $\Delta_{i,1}$ is at most $n(d^{5n}+1).$
\end{lemma}
\begin{proof}
  Assume that 
  \[
  \fB_{s,t}=\fN_{s,t}/{\rm det}(\A) \quad\text{with}\quad \fN_{s,t},\det(\A) \text{~in~} \C[(\frak A_{j,k})_{1\le j,k \le n}],
  \]
  for $1 \le s,t \le n$. Then, by Cramer's formulas, we have $\deg(
  \fN_{s,t}),\deg(\det(\A)) \leq n,$ and since we have cleared all
  denominators by multiplying through with $(\det( \A))^{\deg
    (\gi)+1},$ and guaranteed the presence of an extra factor
  $\det(\A)$, we therefore obtain
  \[
  \deg( \Delta_{i,1}) \leq n\deg( \gi) + n \leq n(d^{5n}+1). \qedhere
  \]
\end{proof}
We first prove that $\Delta_{i,1}$ allows us to control when $F^\mA$
satisfies ${\bf H}_i(1)$.  The main ingredients in the proof of the
following lemma are taken from~\cite{TWT}, with no modification;
this reference itself follows previous work such as~\cite{BaGiHeSaSh10}.
\begin{lemma}
  For $\mA$ in $\C^{n\times n}$, if $\mA$ does not cancel $\D_{i,1}$,
  then $\mA$ is invertible and the polar variety $W(i,F^\mA)$ is
  either empty or $(i-1)$-equidimensional.
\end{lemma}
\begin{proof}
  Consider $\bm A \in \C^{n \times n}$ that does not cancels
  $\D_{i,1}$.  Since $\det(\A)$ divides $\D_{i,1}$, $\bm A$ is
  invertible, and by construction the first $i$ rows $\bm b$ of $\bm
  A^{-1}$ do not cancel $\gi$. We put
  \[
  Y := \left\{\xb \in V(F)~|~\rk (\bm K_i(\bm x,\bm b)) < p+i\right\}. 
  \]
  Lemma B.5 from~\cite{TWT} shows that all irreducible components of
  $Y$ have dimension at least $i-1$; this is essentially Eagon and
  Northcott's result on determinantal varieties~\cite{EN62}, and does
  not depend on our choice of $\bm b$. On the other hand, our
  assumption on $\bm b$ allows us to apply Lemma B.11 from~\cite{TWT},
  which shows that all irreducible components of $Y$ have
  dimension at most $i-1$.  Therefore, $Y$ is either empty or
  ($i-1)$-equidimensional. To conclude the proof, we use the equality
  \[
  Y^{\mA} = W\left(i,F^{\mA}\right),
  \]
  established in the same reference immediately before Lemma~B.10.
\end{proof}
We conclude this section with the second property, ${\bf H}_i(2)$.
\begin{lemma}
  For $\mA$ in $\C^{n\times n}$, if $\mA$ does not cancel $\D_{i,1}$,
  then $0$ is a regular value of the $n+p-i$ polynomials
  $F^\mA,\ [L_1~\cdots~L_p]\cdot \mathrm {jac}(F^\mA,i)$ in the open
  set defined by $(L_1,\dots,L_p) \ne (0,\dots,0)$.
\end{lemma}
\begin{proof}
  Take $\mA$ in $\C^{n \times n}$ so that $\Delta_{i,1}(\mA) \not =
  0$, and let $(\xb,\bm\ell) \in \C^{n+p}$ be a zero of the $n+p-i$
  polynomials $F^\mA$ and $[L_1~\cdots~L_p]\cdot \mathrm {jac}(F^\mA,i)$,
  with $\bm \ell$ non-zero. We have to show that the Jacobian matrix
  of these polynomials has full rank $n+p-i$ at $(\xb,\bm\ell)$.

  We define a vector $\thetab = [\vartheta_1 ~\cdots~\vartheta_i] \in
  \C^i$ by writing $\bm \ell \cdot \mathrm{jac}(F^\mA) = [
    -\vartheta_1 ~\cdots~-\vartheta_i ~0~\cdots~0 ]$ (the trailing
  zeros result from our assumption on $\xb$ and $\bm\ell$). It
  follows that $(\xb,\bm \ell,\thetab)$ cancels the equations
  \begin{equation}\label{eq:poly0}
  F^\mA, \quad [L_1 ~\cdots~ L_p~ T_1 ~\cdots~ T_i] \bbm \mathrm{jac}(F^\mA) \\ \bm I_i~~\bm 0_{i \times (n-i)} \ebm,
  \end{equation}
  where the Jacobian matrix of $F$ is taken with respect to the
  variables $\bm X= X_1,\dots,X_n$. We then post-multiply the
  right-hand matrix by $\mA^{-1}$, and use the fact that ${\rm
    jac}(F^\mA) = {\rm jac}(F)^\mA \, \mA$.  This shows that $(\xb,\bm
  \ell,\thetab)$ also cancels the polynomials
  \begin{equation}\label{eq:poly1}
  F^\mA, \quad [L_1 ~\cdots~ L_p~ T_1 ~\cdots~ T_i] \bbm \mathrm{jac}(F)^\mA \\ \bm b  \ebm,    
  \end{equation}
  where again $\bm b$ denotes the first $i$ rows of $\mA^{-1}$.
  Setting $\xb'=\mA^{-1} \xb$, we deduce that the point $(\xb',\bm\ell,\thetab)$ 
  cancels 
  \begin{equation}\label{eq:poly2}
    F, \quad [L_1 ~\cdots~ L_p~ T_1 ~\cdots~ T_i] \bbm \mathrm{jac}(F) \\ \bm b  \ebm,    
  \end{equation}
  that is, the polynomials $\Phi_{\bm b}$ defined in the preamble.
  The assumption on $\mA$ shows that $0$ is a regular value of this
  mapping in the open set defined by $(L_1,\dots,L_p) \ne
  (0,\dots,0)$.  Since $\bm \ell$ is by definition non-zero, this
  implies that the Jacobian matrix of the polynomials in
  Eq.~\eqref{eq:poly2} has full rank $n+p$ at
  $(\xb',\bm\ell,\thetab)$. Back in the original coordinates, we
  deduce that the Jacobian matrix of the polynomials in
  Eq.~\eqref{eq:poly1} has full rank $n+p$ at
  $(\xb,\bm\ell,\thetab)$. Right multiplication by $\mA^{-1}$
  in~\eqref{eq:poly2} amounts to performing a linear combination of
  the equations; hence, the Jacobian matrix of the polynomials in
  Eq.~\eqref{eq:poly0} has full rank $n+p$ at $(\xb,\bm\ell,\thetab)$
  as well.

  The Jacobian matrix of these polynomials taken with respect to the
  variables $X_1,\dots,X_n$, $L_1,\hdots,L_p$ and $T_1,\hdots,T_i$ is
  equal to
  \begin{align*}
    \left[ 
      \begin{array}{cc}
        \jac(F^{\mA})~~~~~ \bz_{p \times p} & \bz_{p\times i}\\
        \jac_{\bm X,\bm L}\left([\bm L, \bm T ] \cdot 
        \bbm 
        \jac(F^{\mA})\\
        \bm I_i~~\bm 0_{i \times n-i} \\
        \ebm\right) & \bmat \bm I_{i}\\ \bz_{(n-i)\times i} \emat\\
      \end{array}
      \right]
    &=
    \left[ 
      \begin{array}{cc}
        \jac(F^{\mA}) ~~~~ \bz_{p \times p} & \bz_{p\times i} \\
        \ast \ast \ast & \bm I_{i}\\
    \jac_{\bm X,\bm L}\left(\bm L \cdot \jac(F^{\mA},i) \right)& \bz_{(n-i)\times i}
      \end{array}
      \right].
    \end{align*}
    Therefore, after removing $i$ rows and columns, we can see that
    the submatrix
    \begin{align}
    \left[ 
    \begin{array}{c}
    \jac(F^{\mA}) ~~~~ \bz_{p \times p}\\
    \jac_{\bm X,\bm L}\left(\bm L \cdot \jac(F^{\mA},i) \right) 
    \end{array}
    \right]
    \end{align}
    has full rank $n + p-i$ at $(\xb,\bm\ell)$. 
\end{proof}


\section{Genericity of $\textbf{H}_i(3)$}\label{ssec:Hi2}

Notation being as before, we now discuss the last genericity property
that depends on our choice of coordinates. We already showed that in
generic coordinates, the polar variety $W(i, F)$ is either empty or
$(i-1)$-equidimensional. It remains to do the same for $\bm H_i(3)$, 
that is, to prove that if it is not empty, $W(i, F)$ is generically in Noether
position for $\pi_{i-1}$. We will prove the following, where $\D_{i,1}$ is
from Proposition~\ref{prop:Hi12}.
\begin{prop}\label{prop:Hi123}
  For $i=1,\dots,\delta+1$, there exists a non-zero polynomial 
  $\D_{i,2}$ in $\C[(\A_{k,m})_{1 \le k,m \le n}]$ of degree at most
  $4n^2(2d)^{4n}$ such that if $\mA$ does not cancel $\D_{i,1}
  \D_{i,2}$, then $F^{\mA}$ satisfies $\bm H_i(1)$, $\bm H_i(2)$ and
  $\bm H_i(3)$.
\end{prop}
Some results in a similar vein appear in the literature. For instance,
Lemma 5 in~\cite{JeSa10} and Proposition 4.5
in~\cite{SharpEstimatesForTheEffectiveN} are quantitative Noether
position statements. However, our results do not follow from these
previous references, as these previous works analyze the probability
that for a {\em fixed} algebraic set $V$, $V^\mA$ be in Noether
position.  This does not solve our question, since
$W(i,F^\mA)$, which we are interested in, is in general
different from $W(i,F)^\mA$.

Instead, we will rely on the proof given in~\cite{EMP} that
$W(i,F^\mA)$ is in Noether position for a generic $\mA$.  However,
we will not directly analyze the constructions used in that reference,
since they involve e.g. primary decomposition in $\C((\A_{j,k})_{1 \le
  j,k \le n})[X_1,\dots,X_n]$, and the resulting degree bounds would
be way beyond our target. We will instead combine results
from~\cite{EMP} with an effective form of the Nullstellensatz given
in~\cite{EN}; as a result, we have to use Lagrange systems to
describe polar varieties, since systems of minors do not satisfy
assumptions needed to apply this effective Nullstellensatz.
 
The rest of this section is devoted to the proof of this proposition.
From now on, we fix $i$ in $0,\dots,\delta+1$.


\subsection{Preliminaries}

Property ${\bm H}_i(2)$ states that $0$ is a regular value of the
$n+p-i$ polynomials $F,\ [L_1~\cdots~L_p]\cdot \mathrm {jac}(F,i)$ in
the open set defined by $(L_1,\dots,L_p) \ne (0,\dots,0)$; we saw that it
holds in generic coordinates. We start by establishing some
consequences of this fact for the polynomials ${\frkL}(i, F, \bm u)$
of Eq.~\eqref{eqdef:Iil}.
\begin{lemma}\label{prop:RadLagPolarV}
  Suppose that $0$ is a regular value of the $n+p-i$ polynomials
  $F,\ [L_1~\cdots~L_p]\cdot \mathrm {jac}(F,i)$ in the open set
  defined by $(L_1,\dots,L_p) \ne (0,\dots,0)$. Then, for any $\ub =
  (u_1,\hdots,u_p)$ in $\C^p$, the $n+p-i+1$ polynomials
  \begin{equation*}
    {\frkL}(i, F, \bm u) = F,\ [L_1~\cdots~L_p]\cdot \jac(F, i),\ u_1 L_1 + \cdots + u_p L_p -1
  \end{equation*}
  define a radical ideal, either trivial or $(i-1)$-equidimensional.
\end{lemma}
\begin{proof}
  Take $(\xb,\bm \ell)$ in $\C^{n+p}$ that cancels the $n+p-i+1$
  polynomials in~\eqref{eqdef:Iil}. We prove that the Jacobian matrix
  of these equations has full rank $n+p-i+1$ at $(\xb,\bm\ell)$;
  the conclusion then follows from the Jacobian criterion.

  Since $\bm\ell$ cannot be zero, our assumption implies that the
  Jacobian of the polynomials $F$ and $[L_1~\cdots~L_p]\cdot \mathrm
  {jac}(F,i)$ has full rank $n+p-i$ at $(\xb,\bm\ell)$. The conclusion
  therefore holds if ${\rm grad}(u_1 L_1 + \cdots + u_p L_p-1) = [
    {\bm 0}_{1 \times n} ~ u_1~ \cdots ~u_p ]$ is not in the row space
  of this matrix at $(\xb,\bm\ell)$. The Jacobian matrix of $F$ and
  $[L_1~\cdots~L_p]\cdot \mathrm {jac}(F,i)$ is equal to
  \[
  \left[ 
    \begin{array}{cc}
      \jac(F) & {\bm 0}_{p \times p}\\
     *** &  \jac(F,i)^T 
    \end{array}
    \right].
  \]
  Suppose that $[ {\bm 0}_{1 \times n} ~ u_1~ \cdots ~u_p]$ is in the
  row-space of this matrix. Considering the last $p$ columns gives us
  an equality $[u_1 ~ \cdots ~ u_p] = \bm \mu \jac(F,i)^T $, for
  some $\bm \mu$ in $\C^{1 \times (n-i)}$. Right-multiplying by $\bm
  \ell^T \in \C^{p \times 1}$, we obtain $1 = 0$, a contradiction.
\end{proof}

The result carries over to our original polynomials in generic
coordinates. In what follows, just as we defined $F^\mA$ for $\mA$ in
$\C^{n\times n}$, we define $F^{\A}=(f_1^{\A},\hdots,f_p^{\A})$
as \[(f_1(\A\Xb),\hdots,f_p(\A\Xb)) \in \C((\A_{k,m})_{1 \le k,m \le
  n})[X_1,\dots,X_n]^p.\]

\begin{corollary}\label{prop:RadLagPolarVgen}
  For any $\ub = (u_1,\hdots,u_p)$ in $\C^p$, the $n+p-i+1$
  polynomials
  \begin{equation}\label{eqdef:frkLgen}
    {\frkL}(i, F^\A, \bm u) = F^\A,\ [L_1~\cdots~L_p]\cdot \jac(F^\A, i),\ u_1 L_1 + \cdots + u_p L_p -1
  \end{equation}
  define a radical ideal in $\C((\A_{k,m})_{1 \le k,m \le
    n})[X_1,\dots,X_n,L_1,\dots,L_p]$.
\end{corollary}
\begin{proof}
  Proposition~\ref{prop:Hi12} and the previous lemma show that for
  $\mA$ in a Zariski-dense subset of $\C^{n\times n}$, ${\frkL}(i,
  F^\mA, \bm u)$ is radical in $\C[X_1,\dots,X_n,L_1,\dots,L_p]$; as a
  result, this must also be the case for the ideal ${\frkL}(i, F^\A,
  \bm u)$ in $\C((\A_{k,m})_{1 \le k,m \le
    n})[X_1,\dots,X_n,L_1,\dots,L_p]$.
\end{proof}


\subsection{Degree bounds for integral dependence relationships} 

The results in Section~\ref{sec:applications} imply that $F^\A$
satisfies $\bm H_i(1)$, so that $W(i,F^\A)$ is either empty or
equidimensional of dimension $i-1$. We now point out that $F^\A$ also
satisfies $\bm H_i(3)$. In what follows, as in Eq.~\eqref{eq:frkJ}, we
let ${\frkJ}(i,F^\A)$ be the polynomials consisting of $F^\A$ and all
$p$-minors of $\jac(F^\A,i)$ in $\C((\A_{k,m})_{1 \le k,m \le
  n})[X_1,\dots,X_n]$, and we let $\mathscr{K}$ be the ideal they
generate in $\C((\A_{k,m})_{1 \le k,m \le n})[X_1,\dots,X_n]$.  In
particular, the defining ideal of $W(i,F^\A)$ is $\sqrt{\mathscr{K}}$.

Our first lemma simply recalls results from~\cite{EMP}.  In the
following two lemmas, memberships statements are all considered in
$\C((\A_{k,m})_{1 \le k,m \le n})[X_1,\dots,X_n]$; however, in the
course of the proof of Lemma~\ref{lemma:boundP}, we will work with the
same polynomials, but seen in other polynomial rings.
\begin{lemma}\label{lem:6.1}
  For $j=i,\dots,n$, there exists $Q_j$ in $\C((\A_{k,m})_{1 \le k,m
    \le n})[X_1,\dots,X_{i-1},X_j]$, monic in $X_j$ and with $Q_{j}$
  in $\sqrt{\mathscr{K}}$. Furthermore, for any prime component $\frak
  P$ of the ideal $\sqrt{\mathscr{K}}$ in $\C((\A_{k,m})_{1 \le k,m \le
    n})[X_1,\dots,X_n]$, we have $\frak P \cap \C((\A_{k,m})_{1 \le
    k,m \le n})[X_1,\dots,X_{i-1}] = \{0\}$.
\end{lemma}
\begin{proof}
  If $W(i,F^\A)$ is empty, then $\sqrt{\mathscr{K}}$ is the trivial
  ideal, so we simply take $Q_j=1$ for all $j$; the second statement
  is vacuously true.

  Otherwise, let $(\fp_\ell)_{1 \le \ell \leq L}$ be the prime
  components of $\sqrt{\mathscr{K}}$.  By assumption, $L \ge 1$
  and all $\fp_\ell$ have dimension $i-1$. By \cite[Proposition
    1]{EMP}, for all $\ell$,
  \[
    \C((\A_{k,m})_{1 \le k,m \le n})[X_1,\dots,X_{i-1}]\rightarrow\C((\A_{k,m})_{1 \le k,m \le n})[X_1,\dots,X_n]/\fp_\ell
  \] 
  is injective and integral. In particular, this means that $\fp_\ell$
  contains no non-trivial polynomial in $\C((\A_{k,m})_{1 \le k,m \le
    n})[X_1,\dots,X_{i-1}]$, as claimed. Also, it proves that polynomials
  $q_{\ell,j}\in\C((\A_{k,m})_{1 \le k,m \le
    n})[X_1,\dots,X_{i-1},X_j]$ exist, all monic in $X_j$, with
  $q_{\ell,j}\in \fp_\ell$ for each $j$ in $\{i,\hdots,n\}.$
  Thence, \[Q_{j} := \prod_{1 \le \ell\le L} q_{\ell,j}\] is monic in
  $X_j$ and satisfies $ Q_{j} \in \sqrt{\mathscr{K}}$, for each $j
  \in \{i,\hdots,n\}.$
\end{proof}
The former lemma does not directly give us degree bounds on the
polynomials $Q_j$. This is the objective of the next step, where we
control degree with respect to all unknowns involved, $X_1,\dots,X_n$
as well as $\A_{1,1},\dots,\A_{n,n}$.  In this respect, if $P$ is any
polynomial in $\C((\A_{k,m})_{1 \le k,m \le n})[X_1,\dots,X_n]$, we
will let $D \in \C[(\A_{k,m})_{1 \le k,m \le n}]$ be the minimal
common denominator of all its coefficients (defined up to a non-zero
constant in $\C$), and we will write $\overline P := D P$, so that
$\overline P$ is in $\C[(\A_{k,m})_{1 \le k,m \le n},X_1,\dots,X_n]$.
\begin{lemma}\label{lemma:boundP}
  For $j = i,\dots,n$, there exists $P_j$ in $\C((\A_{k,m})_{1 \le k,m
    \le n})[X_1,\dots,X_{i-1},X_j]$, monic in $X_j$, with $P_{j}$ in
  $\sqrt{\mathscr{K}}$ and $\deg(\overline{ P_j})\leq (2d)^{2n}.$
\end{lemma}
\begin{proof} 
  Consider the following ideals: they all have for generators the
  polynomials $\frkJ(i,F^\A)$, that is, $F^\A$ and the $p$-minors of
  $\jac(F^\A, i)$, but they lie in different polynomial rings.
  \begin{itemize}
  \item $\mathscr{J}$ is the ideal generated by $\frkJ(i,F^\A)$ in the
    polynomial ring $\C[(\A_{k,m})_{1 \le k,m \le n},X_1,\dots,X_n]$
    in $n^2+n$ indeterminates;
  \item $\mathscr{K}$, which we already saw is generated by the
    polynomials $\frkJ(i,F^\A)$ in the polynomial ring  $\C((\A_{k,m})_{1 \le k,m \le
      n})[X_1,\dots,X_n]$ in $n$ indeterminates (this is the ideal we
    are mainly interested in);
  \item $\mathscr{M}$ is the ideal defined by the same polynomials,
    but this time in the polynomial ring $\C((\A_{k,m})_{1 \le k,m \le
      n},X_1,\dots,X_{i-1})[X_{i},\dots,X_n]$ in $n-i+1$ indeterminates.
  \end{itemize}
  \noindent{{\em Step 0: Excluding a trivial case.}} Suppose that
  $W(i,F^\A)$ is empty, or equivalently that $\mathscr{K}$ is the
  trivial ideal. In this case, we take $P_j=1$ for all $j$, and we are
  done. Henceforth, we assume that we are not in this situation.

  \smallskip\noindent{{\em Step 1: Defining the minimal polynomial $P_j$.}}  The previous
  lemma shows that every irreducible component of the zero-set
  $W(i,F^\A)$ of $\mathscr{K}$ has dimension $i-1$, and that its image
  by the projection $\pi_i$ is onto. As a result, the extended ideal
  $\mathscr{M}$ has dimension zero, and the ring extension
  \[  \C((\A_{k,m})_{1 \le k,m \le
    n},X_1,\dots,X_{i-1}) \to \C((\A_{k,m})_{1 \le k,m \le
    n},X_1,\dots,X_{i-1})[X_{i},\dots,X_n]/\sqrt{\mathscr{M}}\] is a product
  of finite field extensions. For $j=i,\dots,n$, let then $P_j$ be the
  minimal of $X_j$ in this extension. Then, $P_j$ is in
  $\C((\A_{k,m})_{1 \le k,m \le n},X_1,\dots,X_{i-1})[X_j]$ and is
  monic in $X_j$.

  \smallskip\noindent{{\em Step 2: $P_j$ is polynomial in
      $X_1,\dots,X_{i-1}$.}}  For $j$ as above, the polynomial $Q_j$
  also belongs to $\sqrt{\mathscr{M}}$, so that $P_j$ divides $Q_j$ in
  $\C((\A_{k,m})_{1 \le k,m \le n},X_1,\dots,X_{i-1})[X_j].$ We can
  therefore write
  \begin{align*}
    Q_j &= P_jR_j,~~~~ P_j,R_j \in \C((\A_{k,m})_{1 \le k,m \le n},X_1,\dots,X_{i-1})[X_j].
  \end{align*}
  It then follows by Gauss's lemma that we can write
  \begin{align*}
    Q_j = p_jr_j, ~~~~p_j,r_j \in \C((\A_{k,m})_{1 \le k,m \le n})[X_1,\dots,X_{i-1},X_j],
  \end{align*}
  such that $\mu_j \in \C((\A_{k,m})_{1 \le k,m \le n},X_1,\dots,X_{i-1})$ exists with 
  \[
  P_j = \mu_j p_j,~~~~ R_j = \mu_j^{-1}r_j.
  \]
  Since $Q_j$ is monic in $X_j$, $p_j$ and $r_j$ must also be monic in
  $X_j$, so $\mu_j$ must be the coefficient of the highest degree term
  of $P_j$ in $X_j.$ Since $P_j$ is monic in $X_j$, $\mu_j =1$ and
  hence \[P_j=1\cdot p_j=p_j \in \C((\A_{k,m})_{1 \le k,m \le
    n})[X_1,\dots,X_{i-1},X_j].\] 
  \noindent{{\em Step 3: $P_j$ is  in $\sqrt{\mathscr{K}}$.}}
  It follows that $P_j$ belongs to $\sqrt{\mathscr{M}} \cap
  \C((\A_{k,m})_{1 \le k,m \le n})[X_1,\dots,X_{i-1},X_j]$.
  Equivalently, there exists a non-negative exponent $s$ such that
  $P_j^s$ is in the intersection $\mathscr{M} \cap \C((\A_{k,m})_{1 \le k,m \le
    n})[X_1,\dots,X_{i-1},X_j]$.  Clearing denominators in the
  membership equality in $\mathscr{M}$, this means that there exists
  $D$ non-zero in $ \C[(\A_{k,m})_{1 \le k,m \le
      n},X_1,\dots,X_{i-1}]$ such that $D P_j^s$ is in $\mathscr{K}$,
  and thus in $\sqrt{\mathscr{K}}$.

  By the previous lemma, no prime component of $\sqrt{\mathscr{K}}$
  contains any non-zero polynomial in $ \C[(\A_{k,m})_{1 \le k,m \le
      n},X_1,\dots,X_{i-1}]$. As a consequence, $P_j^s$ is in
  $\sqrt{\mathscr{K}}$, and thus so is $P_j$ itself.
  
  \smallskip\noindent{{\em Step 4: Degree of $\overline{P_j}$.}}  For
  the last step of the proof, the ideal $\mathscr{J}$ is used.  Let
  indeed $Z$ be its zero-set in $\C^{n^2 + n}$. Since all polynomials in
  $F^\A$, resp.\ $\jac(F^\A, i)$, have respective degrees at most $2d$
  in $(\A_{k,m})_{1 \le k,m \le n},X_1,\dots,X_n$, resp. $2d-1$,
  Corollary~\ref{coro:degree} shows that $Z$ has degree at most
  $(2d)^{2n}$.

  Let further $Z'$ be obtained by removing from $Z$ all those
  irreducible components whose projection on the space of coordinates
  $(\A_{k,m})_{1 \le k,m \le n},X_1,\dots,X_{i-1}$ is not dense, and
  let $\mathscr{J}'$ be its defining ideal. It is a routine
  verification that the extension of $\mathscr{J}'$ in the polynomial
  ring $\C((\A_{k,m})_{1 \le k,m \le
    n},X_1,\dots,X_{i-1})[X_{i},\dots,X_n]$ is the radical of
  $\mathscr{M}$. As a result, Theorem~2 in~\cite{DaSc04} implies 
  that the total degree of $\overline{P_j}$ is bounded above by $\deg(Z)$;
  this finishes the proof.
\end{proof}


\subsection{Applying the effective Nullstellensatz}

The previous lemma could allow us to give a quantitative proof that
$\bm H_i(3)$ holds generically, if we were able to bound the degree in
$(\A_{k,m})_{1 \le k,m \le n}$ of the corresponding membership
equality in $\sqrt{\mathscr{K}}$. However, we are not aware of a
suitable effective Nullstellensatz. The best suited one, due to
D'Andrea, Krick and Sombra~\cite{EN}, requires that the number of
generators in the ideal we consider be no more than the ambient
dimension; this is in general not the case for the polynomials
${\frkJ}(i, F^\A)$.

As a result, we will use Lagrange systems instead. For $\bm u$ in
$\C^p$, recall the definition of the Lagrange system ${\frkL}(i, F^\A,
\bm u)$ given in Eq.~\eqref{eqdef:frkLgen}; let further
$\mathscr{L}_{\bm u}$ be the ideal these polynomials generate in
$\C((\A_{k,m})_{1 \le k,m \le n})[X_1,\dots,X_n,L_1,\dots,L_p]$.
Proposition~\ref{prop:projection} gives the inclusion $\sqrt{\mathscr{K}}
\subset \sqrt{\mathscr{L}_{\bm u}}$, and
Corollary~\ref{prop:RadLagPolarVgen} shows that $\mathscr{L}_{\bm u}$ is
a radical ideal, so that we have $\sqrt{\mathscr{K}} \subset
\mathscr{L}_{\bm u}$. As a consequence, the polynomials $P_j$, and
thus $\overline{P_j}$ as well, are in~$\mathscr{L}_{\bm u}$ for
any $\bm u$ in $\C^p$.

Let then $\frak u_1,\dots,\frak u_p$ be new indeterminates, and
consider the ideal $\mathscr{L}_{\bm{\frak u}}$ generated by the polynomials
${\frkL}(i, F^\A, \frak u)$ in the ring of polynomials in
$X_1,\dots,X_n,L_1,\dots,L_p$ over the field of coefficients
$\C((\A_{k,m})_{1 \le k,m \le n},\frak u_1,\dots,\frak u_p)$; the only
difference with the previous setting is that the linear form involved
in these equations is now $\frak u_1 L_1 + \cdots + \frak u_p
L_p-1$. The previous discussion implies that all polynomials
$\overline{P_j}$ belong to $\mathscr{L}_{\bm{\frak u}}$; we are now
going to apply an effective Nullstellensatz to these membership
equalities.

Let $T$ be a new variable, and let $G_1,\dots,G_{n+p-i+1}$ be the
$n+p-i+1$ polynomials in the Lagrange system ${\frkL}(i, F^\A, \frak
u)$.  For $j=i,\dots,n$ applying the
Nullstellensatz in $\C((\A_{k,m})_{1 \le k,m \le n},\frak
u_1,\dots,\frak u_p)[X_1,\dots,X_n,L_1,\dots,L_p,T]$, and clearing
denominators, we obtain the existence of $A_j$ in
$\C[(\A_{k,m})_{1 \le k,m \le n},\frak
u_1,\dots,\frak u_p]-\{0\}$ and of polynomial
coefficients $ C_{j,1},\dots,C_{j,n+p-i+1},B_j$ in $\C[(\A_{k,m})_{1
    \le k,m \le n},\frak
u_1,\dots,\frak u_p)[X_1,\dots,X_n,L_1,\dots,L_p,T]$, such that
\begin{equation}\label{eq:nullst}
A_j = \sum_{\ell=1}^{n+p-i+1} C_{j,\ell} G_\ell + B_j (1-\pjb T).
\end{equation}
Let us then see $A_j$ as a polynomial in $\frak u_1,\dots,\frak u_p$
with non-zero coefficients in $\C[(\A_{k,m})_{1 \le k,m \le n}]$, and let
$\alpha_j$ be one of these coefficients, arbitrarily chosen. We can then
define
\[\D_{i,2}:=\alpha_i \cdots \alpha_n \in \C[(\A_{k,m})_{1 \le k,m \le n}]-\{0\}.\]
With this definition of $\D_{i,2}$, we prove the following lemma.  It
almost completes the proof of Proposition~\ref{prop:Hi123}, except for
the degree bound.
\begin{lemma}\label{lem:6.4}
  If $\mA \in \C^{n\times n}$ does not cancel $\D_{i,1} \D_{i,2}$, then
  $F^{\mA}$ satisfies $\bm H_i(1)$, $\bm H_i(2)$ and $\bm H_i(3)$.
\end{lemma}
\begin{proof}
  Let us take such a matrix $\mA$. The non-vanishing of
  $\D_{i,1}(\mA)$ already guarantees that $F^\mA$ satisfies $\bm
  H_i(1)$ and $\bm H_i(2)$. It remains to establish that $\bm H_i(3)$
  holds, that is, that if it is not empty, $W(i,F^\mA)$ is in Noether
  position for $\pi_{i-1}$. In what follows, we assume that
  $W(i,F^\mA)$ is not empty; by $\bm H_i(1)$, it is
  $(i-1)$-equidimensional.

  Fix $j$ in $i,\dots,n$. Because $\alpha_j(\mA)$ is non-zero, the
  polynomial $a_j:=A_j(\mA, \frak u_1,\dots,\frak u_p)$ is non-zero in
  $\C[\frak u_1,\dots,\frak u_p]$. We choose $\bm u=(u_1,\dots,u_p)$
  in $\C^p$ such that $a_j(u_1,\dots,u_p)$ does not vanish, and such
  that $\bm u$ lies in the open set $\mathscr{O}$ associated by
  Proposition~\ref{prop:projection} to the set $W(i,F^\mA)$.

  Let $g_1,\dots,g_{n+p-i+1}$ be the polynomials in $\frkL(i,F^\mA,\bm u)$.
  Evaluating $(\A_{k,m})_{1 \le k,m \le n}$ at the entries of $\mA$
  and $\frak u_1,\dots,\frak u_p$ at $u_1,\dots,u_p$
  in~\eqref{eq:nullst} gives a relation of the form 
  \[ \tilde a_j = \sum_{\ell=1}^{n+p-i+1} c_{j,\ell} g_\ell + b_j (1-p_j T), \]
  with $\tilde a_j=a_j(u_1,\dots,u_p) \in \C-\{0\}$, polynomials
  $c_{j,\ell}$ and $b_j$ in $\C[X_1,\dots,X_n,L_1,\dots,L_p,T]$ and
  $p_j$ in $\C[X_1,\dots,X_{i-1},X_j]$, monic in $X_j$.

  The next step is routine. Replace $T$ by $1/p_j$ in the previous
  equality; after clearing denominators, this gives a membership
  equality of the form
  \[
  p_j^k \in \langle \frkL(i,F^\mA,\bm u) \rangle
  \]
  for some integer $k \ge 1$ (we cannot have $k=0$, since we assumed
  that $W(i,F^\mA)$ is not empty). Using our assumption on $\bm u$,
  Proposition~\ref{prop:projection} then shows that $p_j^k$ is in the
  ideal generated by ${\frkJ}(i, F^\mA)$, that is, by $F^\mA$ and the
  $p$-minors of $\jac(F^\mA,i)$. In other words, $p_j$ is in the
  defining ideal of the polar variety $W(i,F^\mA)$. Repeating this for
  all $j=i,\dots,n$ proves that $W(i,F^\mA)$ is in Noether position for
  $\pi_{i-1}$.
\end{proof}

\noindent
To estimate the degree of $\D_{i,2}$, what remains is to give an upper
bound on the degrees of $\alpha_i,\dots,\alpha_n$. This will come as
an application of the effective Nullstellensatz given in~\cite{EN}
over the function field $\C((\A_{k,m})_{1 \le k,m \le n},\frak
u_1,\dots,\frak u_p)$. For this, we first need to determine degree
bounds, separately in the actual indeterminates $
X_1,\dots,X_n,L_1,\dots,L_p,T$ and in the ``constants'' $(\A_{k,m})_{1
  \le k,m \le n},\frak u_1,\dots,\frak u_p$, of the polynomials in the
membership relationship; we denote the former by $\deg_{\bm X, \bm L,
  T}$ and the latter by $\deg_{\A,\frak u}$. Then, we have
\begin{align*}
\deg_{\Xb,\Lb,T}
\left\{ G_1,\dots,G_{n+p-i+1}, \frak u_1 L_1 + \cdots + \frak u_p L_p-1\right\}
\leq d, ~
\deg_{\Xb,\Lb,T}(1-T\pjb) \leq (2d)^{2n} +1, 
\end{align*}
and 
\begin{align*}
\deg_{\A,\frak u}  
\left\{  G_1,\dots,G_{n+p-i+1}, \frak u_1 L_1 + \cdots + \frak u_p L_p-1\right\} 
\leq d~\textrm{and}
\deg_{\A,\frak u}(1-T\pjb)& \leq (2d)^{2n}.
\end{align*}
For each $j \in \{i,\hdots,n-p+1\},$ since the number of equations in
the ideal we consider is less than or equal to the ambient dimension $n+p+1$,
it follows from \cite[Theorem 0.5]{EN} that
\[
\deg(A_j)  \le (2n+2)d^{2n+1}((2d)^{2n}+1);
\]
we will use the slightly less precise bound \[\deg(A_j) \le
4n(2d)^{4n}.\] In particular, the same bound holds for the degree
of $\alpha_j$, and this gives
\[\deg(\D_{i,2}) \le 4n^2(2d)^{4n}.\]
This concludes the proof of Proposition~\ref{prop:Hi123}.


\section{Genericity of $\textbf{H}_i^{'}$ and consequences}\label{Sec:Hip}

We still consider a sequence of polynomials $F= (f_1,\hdots,f_p) \in
\C[X_1,\hdots,X_n]^p$ as before; in particular, recall we write
$\delta = n-p$ and that $d$ is an upper bound on the degrees of
$f_1,\dots,f_p$.  Besides, we now also assume that $F$ satisfies all
assumptions $\bm H_i$ (for instance, because we have already applied a
generic change of coordinates), and we prove the following.

\begin{prop}\label{prop:hi2}
 For $i=1,\dots,\delta+1$, if $F$ satisfies $\bm H_i$, there exists a
 non-zero polynomial $\Xi_{i} \in \C[S_1,\dots,S_{i-1}]$ of degree at
 most $d^{4n}$ such that if $ \bm \sigma =
 (\sigma_1,\hdots,\sigma_{i-1}) \in \C^{i-1}$ does not cancel
 $\Xi_{i}$, then $\bm \sigma$ satisfies assumption $\bm H'_i$, that
 is, $0$ is a regular value of the $n+p-1$ polynomials
  \[ X_1-\sigma_1,\dots,X_{i-1}-\sigma_{i-1},\ F,\ [L_1~\cdots~L_p]\cdot \jac(F, i)\]
  in the open set defined by $(L_1,\dots,L_p) \ne (0,\dots,0)$.
\end{prop}
\begin{proof}
  Let $\Psi: \C^{n+p} \times \C^{i-1} \rightarrow
  \C^{n+p-1}$ be the mapping defined by the $n+p-1$ polynomials
  \[
  X_1 - S_1, \dots, X_{i-1} - S_{i-1},\ F,\ [L_1~\cdots~L_p]\cdot \jac(F, i)
  \]
  in indeterminates $X_1,\dots,X_n,L_1,\hdots,L_p,S_1,\dots,S_{i-1}$.
  We claim that $0$ is a regular value of $\Psi$ in the open set $\Omega
  \times \C^{i-1} \subset \C^{n+p} \times \C^{i-1}$, here $\Omega
  \subset \C^{n+p}$ is defined by $(L_1,\dots,L_p) \ne (0,\dots,0)$.
  
  Consider a zero $(\xb,\bm \ell,\bm \sigma)$ of $\Psi$, with $\bm \ell$
  non-zero. Indexing columns by
  \[
  X_1,\dots,X_n,L_1,\hdots,L_p,S_1,\dots,S_{i-1},
  \]
  the Jacobian matrix of $\Psi$ is equal to
  \[
  \left[ 
    \begin{array}{cc}
      \begin{array}{cc}
        \bm I_{i-1}     & \bz_{(i-1)\times (n+p-i+1)}  
      \end{array}  &-\bm I_{i-1}\\
      \jac_{\bm X,\bm L}\left( F,\ \bm L \cdot \jac (F,i)  \right) & \bz_{(n+p-i)\times (i-1)}
    \end{array}
    \right].
  \]
  Because $\bm \ell$ is non-zero, $\bm H_i(2)$ shows that the Jacobian
  matrix $\jac_{(\bm X,\bm L)}\left(F,\ [L_1~\cdots~L_p] \cdot \jac
  (F,i) \right)$ has full rank $n+p-i$ at $(\xb,\bm \ell)$. Hence, the
  entire matrix must have full rank $n+p-1$ at $(\xb,\bm \ell,\bm
  \sigma)$, and $0$ is a regular value of $\Psi$.
  
  Since all polynomials defining $\Psi$ have degree at most $d$, it
  follows by Proposition~\ref{prop:weak_t} that there exists a non-zero
  polynomial $\Xi_{i}$ in $\C[S_1,\dots,S_{i-1}]$ of degree at most
  $d^{(n+p)+(n+p-1)}\leq d^{4n},$ with the property that, if $\Xi_{i}(\bm
  \sigma)\neq 0$ then at any root $(\xb,\bm\ell)$ of the induced 
  mapping $\psi_{\bm\sigma}$ given by
  \begin{equation}\label{eqdef:fiber}
    X_1-\sigma_1,\dots,X_{i-1}-\sigma_{i-1},\ F,\ [L_1~\cdots~L_p]\cdot \jac(F, i),  
  \end{equation}
  if $\bm \ell$ non-zero, then the Jacobian matrix of these equations
  has full rank $n+p-1$ at $(\xb,\bm\ell)$. The proposition is proved.
\end{proof}

We will use this property through the following corollary, which we
already mentioned in Subsection~\ref{ssec:algo}.
\begin{corollary}\label{coro:8}
 For $i=1,\dots,\delta+1$, if $F$ satisfies $\bm H_i$ and $\bm \sigma$
 satisfies $\bm H'_i$, then for any $\bm u=(u_1,\dots,u_p)$
 in $\C^p$, $0$ is a regular value of the $n+p$ polynomials
  \begin{equation*}
    X_1 - \sigma_1, \dots, X_{i-1} - \sigma_{i-1},\ F,\ [L_1~\cdots~L_p]\cdot \jac(F, i),\ u_1 L_1 + \cdots + u_p L_p -1.
  \end{equation*}
\end{corollary}
\begin{proof}
  The proof is similar to that of Lemma~\ref{prop:RadLagPolarV}. Suppose that
  $\Xi_{i}(\bm \sigma)$ is non-zero, let $\bm u = (u_1,\dots,u_p)$ be
  arbitrary in $\C^p$ and take $(\xb,\bm \ell)$ in $\C^{n+p}$ that
  cancels the $n+p$ polynomials
  \begin{equation}\label{eqdef:fiber2} 
    X_1 - \sigma_1, \dots, X_{i-1} - \sigma_{i-1},\ F,\ [L_1~\cdots~L_p]\cdot \jac(F, i),\ u_1 L_1 + \cdots + u_p L_p -1. 
  \end{equation}
  Since $\bm\ell$ is necessarily non-zero, the previous discussion
  implies that the Jacobian of the polynomials in
  Eq.~\eqref{eqdef:fiber} has full rank $n+p-1$ at $(\xb,\bm\ell)$.
  This Jacobian matrix is equal to
  \[
  \left[ 
    \begin{array}{cc}
      \bm I_{i-1} ~ \bm 0_{(i-1) \times (n-i+1)} & \bm 0_{(i-1) \times p}\\
      \jac(F) & {\bm 0}_{p \times p}\\
      *** &  \jac(F,i)^T 
    \end{array}
    \right].
  \]
  As in the proof of Lemma~\ref{prop:RadLagPolarV}, if we suppose that
  $[ {\bm 0}_{1 \times n} ~ u_1~ \cdots ~u_p]$ is in the row-space of
  this matrix, considering the last $p$ columns and multiplying by
  $\bm\ell^T \in \C^{p\times 1}$ leads us to a contradiction. This
  proves that the Jacobian matrix of the equations in
  Eq.~\eqref{eqdef:fiber2} has full rank $n+p$ at $(\bm x,\bm \ell)$,
  as claimed.
\end{proof}


\section{Analysis of the main algorithm}\label{sec:analysis}

We conclude this paper by revisiting the algorithm sketched in
Subsection~\ref{ssec:algo}. The probability analysis is based on the
quantitative genericity results we established in the previous
sections, using the DeMillo-Lipton-Schwartz-Zippel lemma. {\em In order to
simplify some big-O estimates, we assume that the bound $d$ on the
degrees of the input polynomials satisfies $d \ge 2$, since the case
of linear polynomials is trivial.}


\subsection{Description of the pseudocode}

The algorithm is randomized and takes as input a parameter $\epsilon
\in (0,1)$; the choices made in the algorithm guarantee that the
probability of success is at least $1-\epsilon$. 

Randomness occurs in part due to the various choices we make (change
of variables $\mA$, parameter $\bm \sigma$, parameter~$\bm u$).
Besides, at Step 4, we use a minor modification of \cite[Algorithm
  2]{SH} to solve the system
\[
X_1 - \sigma_1, \dots, X_{i-1} - \sigma_{i-1},\ F^{\mA},\ [L_1~\cdots~L_p]\cdot \jac(F^{\mA}, i),\ u_1 L_1 + \cdots + u_p L_p -1.  
\]
of $n+p$ equations in $n+p$ unknowns
$X_1,\dots,X_n,L_1,\dots,L_p$. This subroutine is randomized as well;
in order to guarantee a higher probability of success, we repeat the
calculation $k$ times, for a well-chosen parameter $k$, and keep the
output with the largest cardinality (we discuss this in our
probability analysis below). Upon success, we have obtained a
zero-dimensional parameterization
$\scrQ_i=((q_i,v_{i,1},\dots,v_{i,n+p}),\lambda_i)$ of the solutions
$Z_i$ of these equations, but we are only interested in the projection
$Z'_i$ of these points on the $X_1,\dots,X_n$-space. Recall that
$\scrQ_i$ is such that
\begin{itemize}
\item $\lambda_i(v_{i,1},\dots,v_{i,n+p})=T q_i' \bmod q_i$, with
  $\lambda_i$ a $\Q$-linear form in $X_1,\dots,X_n,L_1,\dots,L_p$;
\item we have the equality $Z_i=\left \{\left(
  \frac{v_{i,1}(\tau)}{q'(\tau)},\dots,\frac{v_{i,n+p}(\tau)}{q'(\tau)}\right
  ) \ \mid \ q(\tau)=0 \right \}.$
\end{itemize}
The only constraint on $\lambda_i$ is that it take pairwise distinct
values on the points of $Z_i$. Now, since the equations defining $Z_i$
are linear in $L_1,\dots,L_p$, the projection $Z_i \to Z'_i$ is
one-to-one; this means that we can take $\lambda_i$ depending on
$X_1,\dots,X_n$ only. This constraint can be enforced at no extra cost
in the algorithm of~\cite{SH}; if this is the case, then
$\scrQ'_i=((q_i,v_{i,1},\dots,v_{i,n}),\lambda_i)$ is a
zero-dimensional parameterization of $Z'_i$.

The algorithm of~\cite{SH} also requires that the input system be
given by a straight-line program. We build it (at Step 3) in the
straightforward manner already suggested in the introduction: given
$F=(f_1,\hdots,f_p)$ in $\C[X_1,\hdots,X_n]^p$, we can build a
straight-line program that evaluates each $f_i$ in $O(d^n)$
operations, by computing all monomials of degree up to $d$,
multiplying them by the corresponding coefficients in $f_i$, and
adding results. To obtain a straight-line program for $f_i^\mA$, we
add $O(n^2)$ steps corresponding to the application of the change of
variables $\mA$. The number of operations here is thus
\[
O(nd^n + n^3) = O^{\sim}(d^n);
\]
note that here, we use the assumption $d \ge 2$.
From this, we can compute and evaluate the required partial
derivatives in the Jacobian of $F^\mA$ in
\[
O(nd^n) = O^{\sim}(d^n)
\]
operations ~\cite{BaSt83}.  Then, the matrix vector product with the
vector of Lagrange multipliers adds a cost that is polynomial in $n$
and which we can therefore neglect in the soft-O notation. Finally, we
add the linear equations $X_1-\sigma_1,\hdots,X_{i-1}-\sigma_{i-1}$;
this gives the straight-line program $\Gamma_i$, whose length is
$O^{\sim}(d^n).$

\begin{algorithm}[!h]
  \KwIn{$F=(f_1,\hdots,f_p) \in \ZZ[X_1,\hdots,X_n]^p$ with $\deg(f_i) \leq d$ and $\htt(f_i) \leq b$, and $0 < \epsilon < 1$. 
    Assume that $d \ge 2$.    }
  \KwOut{$n-p+1$ zero-dimensional parameterizations, the union of whose zeros
    includes at least one point in each connected component of $V(F) \cap \R^n$, with probability at least $1-\epsilon$.} 
  
  \nl Construct \[S := \{1,2,\hdots,\lceil 4\epsilon^{-1}5n^3(2d)^{5n} \rceil\},\] \[T :=
  \{1,2,\hdots,\lceil 4\epsilon^{-1}nd^{4n} \rceil \},\] 
  \[R :=
  \{1,2,\hdots,\lceil 4\epsilon^{-1}nd^{2n} \rceil \},\]
  and  choose  $\mA \in
  S^{n^2}$, $\bm \sigma \in T^{n-1}$ and $\ub \in R^p$ uniformly at random\; 
  
  \caption{{Main Algorithm} \label{alg:1}} 
  
  \nl \For{$i\gets1$ \KwTo $n-p+1$}{
    \nl Build a straight-line program $\Gamma_i$ that computes the equations
    \begin{equation}\label{eq:step3}
      X_1 - \sigma_1, \dots, X_{i-1} - \sigma_{i-1},\ F^{\mA},\ [L_1~\cdots~L_p]\cdot \jac(F^{\mA}, i),\ u_1 L_1 + \cdots + u_p L_p -1;        
    \end{equation}
    \nl Run \cite[Algorithm 2]{SH} $k \geq \log_2(4n/\epsilon)$ times
    with input $\Gamma_i$\;
    
    \nl Let $\mathscr{Q}_i=((q_i,v_{i,1},\dots,v_{i,n+p}),\lambda_i)$ be the highest cardinality
    zero-dimensional parameterization returned in Step 4\; 
    
    \nl Denote by $\mathscr{Q}_i^{'}=((q_i,v_{i,1},\dots,v_{i,n}),\lambda_i)$ the parameterization of the projection of $\mathscr{Q}_i$ onto the $X_1,\dots,X_n$-space\;
    
  } 
  \nl  \Return $[\mathscr{Q}^{'}_1,\hdots,\mathscr{Q}^{'}_{n-p+1}]$.
\end{algorithm}
As we already pointed out in Subsection~\ref{ssec:algo}, if $F^\mA$
satisfies ${\bm H}_i$, $\bm \sigma $ satisfies $\bm H_i^{'}$ and $\bm
u$ satisfies $\bm H_i''$, and if $\scrQ_i$ is a zero-dimensional
parametrization of the solutions of the equations~\eqref{eq:step3} at Step 3 (for all
$i \in \{1,\hdots,n-p+1\}$), Theorem~2 in~\cite{EMP} establishes that
the output returned in Step 7 will contain one point in each connected
component of $V \cap \R^n$. (The claim made in
Subsection~\ref{ssec:algo} relied on an assertion that has since
been proved, in Corollary~\ref{coro:8}.)


\subsection{Bit operation cost} 

The following lists the costs for each step of Algorithm \ref{alg:1},
assuming that the input polynomials have degree $d$ and integer
coefficients of height at most $b$.

\smallskip\noindent
(1) We defined $S := \{1,2,\hdots,\lceil 4\epsilon^{-1}5n^3(2d)^{5n} \rceil \}$
and therefore the height of any $a_{i,j} \in S$ is at most
\[
\log (4/{\epsilon}) + \log(5n^3(2d)^{5n}) \in O^{\sim}(\log (1/{\epsilon}) + n\log d).
\]
Since $|R|, |T| \le |S|,$ we also have that the height of any
$\sigma_{k} \in T$ and $u_{\ell} \in R$ admits the same upper bound.

\smallskip\noindent (3) After computing the partial derivatives in the
Jacobian matrix, the height grows by at most another factor of $\log
d$. Thus, all polynomials in the system considered at Step 3 have
height
\[
O^{\sim}(b + d(\log (1/\epsilon) + n\log(d)))
=
O^{\sim}(b + d\log (1/\epsilon) + dn).
\]
All integer coefficients appearing in the straight-line program
$\Gamma_i$ satisfy the same bound.

\smallskip\noindent (4) As a result, after applying \cite[Algorithm
  2]{SH} $k$ times for each index $i$, with $k = O(\log(n) + \log( 1 /\epsilon))$, the total boolean cost of the algorithm is
  \[
O^{\sim}(d^{3n+2p+1}\log(1/\epsilon)(b + \log(1/\epsilon)))
  \]
      where the polynomials in the output have degree at most $d^{n+p},$ and height at most
  \[
O^{\sim}(d^{n+p+1}(b + \log(1/\epsilon))).
  \]
This proves the runtime estimate, as well as our bounds on the height
of the output.


\subsection{Probability of success} 

Let $\Delta_{i,1}$ and $\Delta_{i,2} \in \C[(\frak A_{j,k})_{1\le j,k
    \le n}]$ be the polynomials from Propositions~\ref{prop:Hi12}
and~\ref{prop:Hi123}. Denote by $\Delta := \prod_{i=1}^{n-p+1}
\D_{i,1}\D_{i,2},$ and note that
\begin{align}
    \deg( \Delta) = \sum_{i=1}^{n-p+1} \deg( \D_{i,1})+\deg(\D_{i,2})
    \leq 5n^3(2d)^{5n}.
\end{align}
If $\mA \in \C^{n \times n}$ does not cancel $\Delta,$ then $\mA$ is
invertible and $F^\mA$ satisfies $\bm H_i$ for all $i$ in
$\{1,\hdots,n-p+1\}.$ Now, assuming that $\mA$ is such a matrix, let
$\Xi_i\in \C[S_1,\dots,S_{i-1}]$ be the polynomials from
Proposition~\ref{prop:hi2} applied to $F^{\mA}.$ Denote by $\Xi :=
\prod_{i=1}^{n-p+1} \Xi_i,$ and note that
\begin{align}
    \deg( \Xi)= \sum_{i=1}^{n-p+1} \deg( \Xi_i) \leq nd^{4n}.
\end{align}
If $\bm \sigma \in \C^{i-1}$ does not cancel $\Xi$, then it satisfies
$\bm H_i^{'}$ for all $i \in \{1,\hdots,n-p+1\}.$ Assume that this is
the case.

As argued in Subsection~\ref{ssec:algo}, the last condition on our
parameters is that $\bm u$ satisfy $\bm H_i''$ for all $i$. For a
given index $i$, Proposition~\ref{prop:projection} shows the existence
of a non-zero polynomial $\Upsilon_i$ in $\C[U_1,\dots,U_p]$ such that
if $\Upsilon_i(u_1,\dots,u_p)$ is non-zero, $\bm H_i''$ holds; in
addition, that proposition and Corollary~\ref{coro:degree} give an
upper bound of $d^{n+p}$ for the degree of $\Upsilon_i$. We denote by
$\Upsilon := \prod_{i=1}^{n-p+1} \Upsilon_i,$ and note that
\begin{align}
    \deg( \Upsilon) = \sum_{i=1}^{n-p+1} \deg( \Upsilon_i) \leq nd^{2n}.
\end{align}
If $\ub \in \C^{p}$ does not cancel $\Upsilon$, then $\ub$ satisfies $\bm H_i^{''}$ for all $i \in
\{1,\hdots,n-p+1\}.$

Then, the algorithm is guaranteed to succeed, as long as our calls to
Algorithm 2 in~\cite{SH} succeed in solving the equations at
Step~3. That reference establishes that by repeating the calculation
$k$ times, and keeping the output of highest degree among those $k$
results, we succeed with probability at least $1-(1/2)^k$. When
Algorithm 2 does not succeed, it either returns a proper subset of the
solutions, or FAIL. Note that Algorithm 2 is shown to succeed in a
single run with probability at least $1-11/32,$ and we bound the
probability of success with $1-1/2$ for simplicity.  Now, 
recall that we choose
 $\mA$ in $S^{n^2}$, $\bm \sigma$ in $T^{n-1}$ and $\ub$ in $R^{p}$ 
uniformly at random, with
\begin{align*}
S &= \{1,2,\hdots,\lceil 4\epsilon^{-1}5n^3(2d)^{5n}\rceil \},\\
T &= \{1,2,\hdots,\lceil 4\epsilon^{-1}nd^{4n} \rceil \},\\
R &= \{1,2,\hdots,\lceil 4\epsilon^{-1}nd^{2n} \rceil \}.
\end{align*}
Using the DeMillo-Lipton-Schwartz-Zippel lemma, we obtain
\[
\pr[\Delta(\mA)=0] \leq  \frac{\deg\Delta}{|S|} = \epsilon/4.
\]
If this is the case, then 
\[
\pr[\Xi(\bm \sigma)=0] \leq  \frac{\deg\Xi}{|T|} = \epsilon/4,
\]
and if this is the case, then
\[
\pr[\Upsilon(\ub)=0] \leq  \frac{\deg\Upsilon}{|R|} = \epsilon/4.
\]
When all this holds, for a given index $i$, Step~4 succeeds with
probability at least $1-1/2^k$, so the probability that all indices
$i$ succeed is at least $(1-1/2^k)^n$; our choice of the parameter $k$
at Step~4 ensures that this probability is at least $\epsilon/4$ as
well. Therefore, the overall probability of success is
at least
\[(1-\epsilon/4)^4 \ge 1-\epsilon.\]
This finishes the proof of Theorem~\ref{theo:main}.




\bibliographystyle{plain}
\bibliography{refs.bib}

\end{document}